\documentclass[preprint,12pt]{elsarticle}

\usepackage{amsmath,amssymb,color}
\usepackage{graphicx,epsfig,framed}
\usepackage{mathptmx,times} % if new font selection scheme installed
\usepackage{epstopdf,textcase} % flushend.sty not reachable  
\usepackage[latin1]{inputenc}  
\usepackage{soul,multirow,pifont} % colortbl.sty not reachable
\usepackage{color,alltt}          % pbox.sty not reachable
\usepackage[hidelinks]{hyperref}
\usepackage{enumerate,siunitx}    % breakurl.sty not reachable
\usepackage{bigints}
\usepackage{tikz}
\usetikzlibrary{shapes.geometric}
\usepackage{circuitikz}

\usepackage{natbib}
\newtheorem{theorem}{Theorem}[section]

\newtheorem{definition}[theorem]{Definition}

\newtheorem{remark}[theorem]{Remark}
\newcommand{\rfb}[1]{\mbox{\rm
   (\ref{#1})}\ifx\undefined\stillediting\else:\fbox{$#1$}\fi}

\newfont{\roma}{cmr10 scaled 1200}

\newcommand{\rline}  {{\mathbb R}}

\newcommand{\BBB}  {{\bf B}}
\newcommand{\CCC}  {{\bf C}}
\newcommand{\EEE}  {{\bf E}}
\newcommand{\JJJ}  {{\bf J}}

\newcommand{\HHH} {{\mathbf H}}

\newcommand{\PPP} {{\mathbf P}}
\newcommand{\SSS} {{\mathbf S}}
\newcommand{\DDD} {{\mathbf D}}

\newcommand{\dd}  {{\rm d}\hbox{\hskip 0.5pt}}
\renewcommand{\leq} {\leqslant}
\renewcommand{\geq} {\geqslant}

\newcommand{\Dscr} {{\cal D}}
\newcommand{\Escr} {{\cal E}}

\newcommand{\Hscr} {{\cal H}}
\newcommand{\Kscr} {{\cal K}}
\newcommand{\Lscr} {{\cal L}}
\newcommand{\Mscr} {{\cal M}}
\newcommand{\Nscr} {{\cal N}}
\newcommand{\Sscr} {{\cal S}}
\newcommand{\Uscr} {{\cal U}}
\newcommand{\Yscr} {{\cal Y}}
\newcommand{\mm}    {{\hbox{\hskip 0.5pt}}}
\newcommand{\m}     {{\hbox{\hskip 1pt}}}
\newcommand{\n}     {{\hbox{\hskip -5pt}}}
\newcommand{\nm}    {{\hbox{\hskip -3pt}}}

\newcommand{\bluff} {{\hbox{\raise 15pt \hbox{\mm}}}}
\newcommand{\sbluff}{{\hbox{\raise 10pt \hbox{\mm}}}}
\newcommand{\Om}    {{\Omega}}
\renewcommand{\div} {{\rm div\,}}
\newcommand{\e}     {{\varepsilon}}
\renewcommand{\L}   {{\Lambda}}
\renewcommand{\l}   {{\lambda}}

\newcommand{\FORALL}{{\hbox{$\hskip 10mm \forall \;$}}}
\newcommand{\rarrow}{\mathop{\rightarrow}}

\newcommand{\AB}     {{A\& B}}
\newcommand{\CD}     {{C\& D}}

\makeatletter
\renewcommand{\p@enumii}{}
\makeatother

\newcommand{\ULax}   {\boldsymbol{{\mathfrak{T}}}}
\newcommand{\GothA}  {\boldsymbol{{\mathfrak{A}}}}

\newcommand{\rot}    {{\rm rot\m}}
\renewcommand{\div}  {{\rm div\m}}
\newcommand{\bbm}[1]{\left[\begin{matrix} #1 \end{matrix}\right]}
\newcommand{\sbm}[1]{\left[\begin{smallmatrix} #1
\end{smallmatrix}\right]}
\makeatletter
\renewcommand{\p@enumii}{}
\makeatother

\renewcommand{\ULax}   {\boldsymbol{{\mathfrak{T}}}}
\renewcommand{\GothA}  {\boldsymbol{{\mathfrak{A}}}}

\definecolor{forestgreen}{rgb}{0.13, 0.55, 0.13}
\journal{System and Control Letters}

\begin{document}

\begin{frontmatter}
\title{Projected incrementally scattering passive systems on closed
convex sets}
%\title{This is a specimen title\tnoteref{t1,t2}}
\tnotetext[t1]{This research has been funded by the Israel Science 
Foundation under grants no. 3621/21 and 2802/21. At the time of 
writing this article, all authors were affiliated with the School of 
Electrical Engineering, Tel Aviv University, Israel. E-mail 
addresses\m: $^1$shantanu.singh46@gmail.com, 
$^2$sebastien.fueyo@gmail.com, $^3$gweiss@tauex.tau.ac.il\m.}

\author{Shantanu Singh\corref{cor1}$^1$, S\'ebastien Fueyo$^2$, 
George Weiss$^3$}%}
%\ead{shantanu.singh46@gmail.com, sebastien@tauex.tau.ac.il, 
%gweiss@tauex.tau.ac.il}
%\author[2]{S\'ebastien Fueyo}
%\ead{sebastien@tauex.tau.ac.il}
%\author[3]{George Weiss}
%\ead{gweiss@tauex.tau.ac.il}
\address{Department of Electrical and Computer Engineering, Technical
   Univ. of Crete, Crete, Greece.$^1$ Univ. Grenoble Alpes, CNRS, 
   Inria, Grenoble INP, GIPSA-lab, Grenoble, France.$^2$\\ School of
   Electrical Engineering, Tel Aviv University, Ramat Aviv,  Tel Aviv,
   Israel.$^{3}$}
%\affiliation[2]{organization={Tel Aviv University},
%addressline={Ramat Aviv},
%postcode={6997801},
%city={Tel Aviv},
%country={Israel}}
%\affiliation[3]{organization={Tel Aviv University},
%addressline={Ramat Aviv},
%postcode={6997801},
%city={Tel Aviv},
%country={Israel}}

%\thanks{This research has been funded by the Israel Science 
%Foundation under grants no. 3621/21 and 2802/21. The authors are with
%the School of Electrical Eng., Tel Aviv University, Israel. E-mail 
%addresses\m: $^1$shantanu@tauex.tau.ac.il, 
%$^2$sebastien@tauex.tau.ac.il, $^3$gweiss@tauex.tau.ac.il\m.}    

%%%%%%%%%%**********%%%%%%%%%%**********%%%%%%%%%%**********%%%%%%%%%%

%\maketitle
%\thispagestyle{empty}
%\pagestyle{empty}

%%%%%%%%%%**********%%%%%%%%%%**********%%%%%%%%%%**********%%%%%%%%%%
\begin{abstract}
In this article we show that the projected dynamical system obtained
by restricting the state of an incrementally scattering passive system
to a closed and convex subset $K$ of the state space (a real Hilbert
space), is also an incrementally scattering passive system. First we
show that the projection of a maximal dissipative operator to the
tangent cones of $K$ is again maximal dissipative, hence, it
determines a contraction semigroup. Using this result, we prove our
earlier claim. Our results are based on the Crandall-Pazy theorem,
Rockafellar's theorem on sums of operators and Moreau's decomposition 
theorem. We give an application of our results to Maxwell's equations 
on a cylindrical domain, approximately describing a fault current 
limiter, restricting the average current through the cylinder (in the 
direction of its axis) so that its absolute value cannot exceed a 
given threshold.
\end{abstract}


\end{frontmatter}

%%%%%%%%%%**********%%%%%%%%%%**********%%%%%%%%%%**********%%%%%%%%%%
\section{Introduction} \label{sec1}

This paper is about infinite dimensional nonlinear systems that are
incrementally scattering passive. The intuitive meaning of passivity
is that the system does not have an internal source of energy. Passive
systems with a finite dimensional state space are extensively explored
in the textbook \cite[Chapter 4]{van der} and the article
\cite{Willems}. The notion of passivity has paved the way for the
valuable concept of port-Hamiltonian systems, introduced in
\cite{Maschke_van}. For passive infinite dimensional systems, we refer
to \cite{jacob2012linear}, \cite{Ortega_Maschke, Staf_2002},
\cite[Chapter 11]{Staf_book} and \cite{TuWe_survey}, and for
incrementally scattering passive infinite dimensional nonlinear
systems, we refer to \cite{ShWeTu:art149, ShWe:art157, ShWe:art158}.

Following \cite{ShWe:art158}, any incrementally scattering passive
system $\Sigma$ with the state space $X$, the input space $U$ and the 
output space $Y$ (all real Hilbert spaces) can be described locally 
in time as follows: For $[x(t)\ u(t)]^\top\in\Dscr([\AB])$ and $u$ 
in the Sobolev space $\Hscr^1((0,\infty);U)$, \vspace{-1mm}
\begin{equation} \label{Bruno}
  \left\{\begin{array}{ll}\frac{\dd^+}{\dd t} x(t) \m\in \m [\AB]
  \bbm{x(t)\\u(t)}  \FORALL t\geq 0,\vspace{2mm}\\
  y(t) \m=\m [\CD]\bbm{x(t)\\u(t)} \m \FORALL t\geq 0\m. 
  \end{array}\right. \vspace{-1mm}
\end{equation} 
Here $x(t)\in X$ is the state of the system at time $t$, $\frac{\dd^+}
{\dd t}x(t)$ is its right derivative, $u(t)\in U$ is the input signal
at $t$ and $y(t)\in Y$ is the output signal at $t$. The operator
$[\AB]: \Dscr([\AB])\rarrow X$ is nonlinear (possibly set-valued) and
the operator $[\CD] : \Dscr([\AB])\rarrow Y$ is nonlinear and single
valued. The domain $\Dscr([\AB])$ is dense in $X\times U$. The
incremental passivity of this system implies that the state
trajectories $x_1,\m x_2$ and the output functions $y_1,\m y_2$
corresponding to the input functions $u_1,\m u_2$, respectively, that
are smooth enough so that they are solutions of \rfb{Bruno}, satisfy
the following {\it power balance inequality}\m: For all $t\geq 0$
\begin{equation*} 
   \frac{\dd^+}{\dd t}\|x_1(t)-x_2(t)\|^2 \m\leq\m \|u_1(t)-u_2(t)
   \|^2 - \|y_1(t)-y_2(t)\|^2 \m,
\end{equation*}
where $\frac{\dd^+}{\dd t}$ (again) denotes the right derivative at
$t$.

{\bf Our aim} in this article is to investigate the well-posedness and
incremental scattering passivity of the modified system $\Sigma^K$
that is obtained from $\Sigma$ when the state $x(t)$ is constrained to
a closed and convex subset $K\subset X$ with a nonempty interior. When 
the state of \m $\Sigma^K$ is in the interior of $K$, then the 
equations of \m $\Sigma^K$ are the same as those of $\Sigma$. When the 
state $x(t)$ of $\Sigma^K$ reaches the boundary of $K$, then its
velocity $\frac{\dd^+}{\dd t}x(t)$ is constrained to the tangent cone 
to $K$ in that point. Thus, the constrained system $\Sigma^K$ is 
described by the following projected differential inclusion and output
equation\m:
\begin{equation} \label{Rashford}
  \left\{\begin{array}{ll}\frac{\dd^+}{\dd t}x(t) \m\in\m \Pi_K
  \left(x(t),[\AB]\sbm{x(t)\\u(t)}\right)  \hspace{-0.5cm}\FORALL 
  t\geq 0,\vspace{2mm}\\ y(t) \m=\m [\CD]\bbm{x(t)\\u(t)}\hspace{1cm} 
  \FORALL t\geq 0. \end{array}\right. 
\end{equation} 
Here $\Pi_K(p,S)$, for any $p\in K$ and any set $S\subset X$, is the 
projection of $S$ on the tangent cone $T_K(p)$, see Section \ref{sec2}
for more details. The operator $\Pi_K$ projects the set $[\AB]
\sbm{x(t)\\u(t)}$ of possible velocity vectors to the tangent cone 
$T_K(x(t))$ of $K$, thereby forcing $x(t)$ to remain in $K$.

Our work is motivated by physics and engineering applications where
the system is allowed to operate within a convex set due to physical
constraints. One example is a vibrating string whose displacement is
constrained at one interior point (see \cite{singh2023CDC}). Another
example would arise if we consider the model of a flexible wind 
turbine tower coupled with a tuned mass damper (TMD) as in 
\cite[Section 3]{ShWeTu:art149}, and impose a natural constraint on 
the displacement of the TMD within the nacelle of the wind turbine 
tower.

The class of systems represented by \rfb{Rashford} or special cases of
\rfb{Rashford} are often called {\it projected dynamical systems}. 
Relevant references on projected dynamical systems with a finite
dimensional state space are \cite{brog,nagurney1995projected,
brogliato2006equivalence,heemels2023existence,art122,xia2000stability,
hauswirth2018time,gao2003exponential}. In \cite{heemels2023existence,
art122} and \cite{Lorenzetti2022}, the system under consideration is a
feedback system and the constraint is imposed on the state of the 
controller (which, of course, is a part of the overall state). In 
\cite{cojo1}, nonlinear systems on real Hilbert spaces are considered 
with the nonlinear operator $[\AB]$ in \rfb{Bruno} and \rfb{Rashford} 
replaced with a Lipschitz continuous function of the state (hence with
an equality instead of an inclusion in \rfb{Bruno} and 
\rfb{Rashford}). Inputs and outputs are not considered in 
\cite{cojo1}. The class of nonlinear systems as in \rfb{Bruno} are an
extension of the class of systems in \cite{cojo1}, in the sense that 
we allow the states of the systems to evolve according to a 
differential inclusion, moreover, we consider systems with inputs and
outputs.

Our approach is to first represent the unconstrained incrementally
scattering passive system via its nonlinear Lax-Phillips semigroup
(for background on this semigroup see for instance \cite{Staf_2002,
Staf_book, WeSt13} for linear systems and \cite{ShWeTu:art149,
ShWe:art157, ShWe:art158} for nonlinear systems). We denote by
$\GothA$ the maximal dissipative (possibly set-valued) operator that
determines the Lax-Phillips semigroup (see Section \ref{sec2} for
details). Using the projection of $\GothA$ to certain tangent cones,
we show that the new (possibly set-valued) operator representing the
projected dynamical system determines a contractive Lax-Phillips
semigroup of nonlinear operators.

This approach requires inestigating the following Cauchy problem on 
a real Hilbert space $Z$, for some maximal dissipative $\GothA$ and 
a closed and convex $\Kscr\subset \overline{\Dscr(\GothA)}$, with int
\m $\Kscr\not=\emptyset$\m: Find a continuous and right differentiable
$z:[0,\infty)\rarrow Z$ such that $z(t)\in \Dscr(\GothA)\cap\Kscr$
for all $t\geq 0$ and
\begin{equation} \label{Onana}
   \frac{\dd^+}{\dd t}z(t) \m\in\m \Pi_\Kscr(z(t),\GothA z(t))\m,
   \quad z(0) \m=\m z_0\in\Dscr(\GothA)\cap \Kscr.
\end{equation}
We prove in Section \ref{sec3} that the above Cauchy problem has 
well-posed solutions (as defined in Definition \ref{Solution_def}) and
these determines a contraction semigroup $\ULax^\Kscr$ on $Z$. This
result is the basis for proving our main result concerning the
constrained systems of type \rfb{Rashford}.

The outline of the paper is as follows. We recall some facts about
nonlinear (possibly set-valued) dissipative operators and contraction
semigroup of nonlinear operator in Section \ref{sec2}. In Section
\ref{sec3} we prove that the Cauchy problem \rfb{Onana} has a unique
solution (given by a contraction semigroup). In Section \ref{sec4} we
investigate the particular case of the differential inclusion
\rfb{Onana} when $\Kscr$ is generated by a finite number of scalar
constraints (inequalities satisfied by continuous linear functionals
of the state). We introduce incrementally scattering passive nonlinear
systems in Section \ref{sec5} and recall some relevant results about
such systems. In Section \ref{sec6} we show that the projected
dynamical nonlinear systems (described by \rfb{Rashford}) associated
with incrementally scattering passive systems (described by
\rfb{Bruno}) are incrementally scattering passive. An application of
our results to a string equation is given in Section \ref{sec7} and to
Maxwell's equations in Section \ref{sec8}. The example from Section
\ref{sec8} is a model of a fault current limiter. We consider
Maxwell's equations on a bounded cylinder. Using our result from
Section \ref{sec6}, we show that the projected dynamical system
obtained from Maxwell's equations, as a result of forcing the average
current through the cylinder to be less than a given highest allowed
value, is an incrementally scattering passive system.

%%%%%%%%%%**********%%%%%%%%%%**********%%%%%%%%%%**********%%%%%%%%%%
\section{Contraction semigroups of nonlinear operators} \label{sec2}

In this section we recall some preliminaries about monotone operators
on real Hilbert spaces and the nonlinear operator semigroups defined
using such operators. The following results are based on
\cite{Barbu,Brezis,Brezis2,Browder68,CranPazy,Kato_accr,Minty,Rocka}.
For more recent results see \cite{Bauschke,Show}.

Let $Z$ be a real Hilbert space. A nonlinear (possibly set-valued)
operator $\Nscr$ defined on $\Dscr(\Nscr) \subset Z$, whose values are
nonempty subsets of $Z$, is {\em monotone} if for every $w_1,w_2\in
\Dscr(\Nscr)$, the inequality
\begin{equation*} \label{monotone_op}
   \m\langle g_1-g_2, w_1-w_2\rangle \m\geq\m 0 
   \quad\forall g_1\in\Nscr(w_1),\m g_2\in\Nscr(w_2)
\end{equation*}
holds. Such an operator is called {\it maximal monotone} if $\Nscr$
has no proper monotone extension (mapping a subset of $Z$ to subsets
of $Z$). If $\Nscr$ is (maximally) monotone then $-\Nscr$ is called
(maximally) {\it dissipative}. If $\Nscr$ is maximal dissipative, then
its graph is closed, and for every $w\in\Dscr(\Nscr)$, the set 
$\Nscr(w)$ is closed and convex. The closure of $\Dscr(\Nscr)$ is also
a convex set.  

Given two maximal dissipative operators $\Nscr_1$ and $\Nscr_2$, their
sum $\Nscr_1+\Nscr_2$ is not necessarily maximal dissipative as it
might happen that $\Dscr(\Nscr_1)\cap \Dscr(\Nscr_2)$ is too small or
even empty. In this regard, an important theorem to determine the
maximality of the sum of two maximal dissipative operators was given
by R.T. Rockafellar, see \cite[Theorem 1]{Rocka}. Here is a simplified 
statement of this result:

{\color{blue}
\begin{theorem}[Rockafellar] \label{Rockafellar} 
If $\Nscr_1$ and $\Nscr_2$ are maximal dissipative operators on
$\Dscr(\Nscr_1)\subset Z$ and $\Dscr(\Nscr_2)\subset Z$ respectively,
then the operator $\Nscr_1+\Nscr_2$ (defined on $\Dscr(\Nscr_1)\cap
\Dscr(\Nscr_2)$) is maximal dissipative if
\begin{equation*} \label{Rocka_eq}
   \Dscr(\Nscr_1)\cap{\rm int}(\Dscr(\Nscr_2))\m\neq\m \emptyset,
\end{equation*}
where ${\rm int}(\Dscr(\Nscr_2))$ means the interior of 
$\Dscr(\Nscr_2)$. 
\end{theorem}}
\smallskip

A {\em strongly continuous semigroup of nonlinear operators} 
$\ULax=\left(\ULax_t\right)_{t\geq 0}$ acting on a closed subset 
$\Kscr\subset Z$ is defined exactly as in the linear case, without 
requiring that the operators are linear. If $\ULax$ is such a
semigroup, then define the operator \vspace{-2mm} 
\begin{equation} \label{operator_minimal}
   \GothA^0 z \m=\m \lim_{t\rarrow 0,\m\m t>0} \frac{1}{t} \left[ 
   \ULax_t z - z\right] \m,
\end{equation} \vspace{-2mm}
\begin{equation} \label{D(operator_minimal)} 
   \Dscr(\GothA^0) \m=\m \left\{ z\in \Kscr\ |\ \mbox{the above limit 
   exists} \right\} \m.
\end{equation} 
Following \cite{CranPazy}, $\GothA^0$ is called the {\em (strong)
generator} of $\ULax$. Very little is known about semigroups of
nonlinear operators at this level of generality. However, there is a
rich body of knowledge about a subclass of such semigroups, those
that are contractive. The semigroup $\ULax$ is called {\em
contractive} if
$$ \|\ULax_t z_1 - \ULax_t z_2\| \m\leq\m \|z_1-z_2\| 
   \n\n\FORALL z_1,z_2\in \Kscr, \m\m t\geq 0.$$
For the basics about such operator semigroups we refer to
\cite{Brezis74, CranPazy, Kato_accr, Show}. The following theorem
follows from Theorems 1.3 and A2 as well as Lemma 3.2, and
Corollary 3.1 in Crandall and Pazy \cite{CranPazy}.  

{\color{blue}
\begin{theorem} \label{Putin}
Assume that $\ULax$ is a contractive semigroup of nonlinear operators
on a closed and convex subset $\Kscr$ of the Hilbert space $Z$. Then
the operator $\GothA^0$ from \rfb{operator_minimal} and
\rfb{D(operator_minimal)} is densely defined in $\Kscr$ and is
dissipative. The operator $\GothA^0$ has a unique maximal dissipative
extension $\GothA$ (which may be set-valued) with the same domain
$\Dscr(\GothA)=\Dscr(\GothA^0)$. If $z_0\in \Dscr(\GothA)$, then
$\GothA^0z_0$ is the unique element of smallest norm in the closed and
convex set $\GothA z_0$.

Let $z_0\in\Dscr(\GothA)$. The function $z:[0,\infty)\rarrow \Kscr$
defined by $z(t)=\ULax_tz_0$ is Lipschitz continuous and right
differentiable at every $t\geq 0$. Moreover, for every $t\geq 0$, it
holds that $z(t)\in \Dscr(\GothA)$,\vspace{-2mm}
\begin{equation} \label{d^+z/dt}
   \frac{\dd^+}{\dd t}z(t) \m=\m \GothA^0 z(t),\vspace{-2mm}
\end{equation}
and $\GothA^0 z$ is right continuous at $t$. Moreover $z$ is the
unique Lipschitz continuous and right differentiable
$\Dscr(\GothA)\cap\Kscr$-valued function that is a solution of
\rfb{d^+z/dt} and $z(0)=z_0$. The function $z$ is also the unique
Lipschitz continuous and right differentiable function that satisfies
$z(0)=z_0$ and the differential inclusion
$$\frac{\dd^+}{\dd t}z(t)\in \GothA z(t)\m.$$
\end{theorem}}
\smallskip

{\it Proof:\m\m} The claim that the strong generator $\GothA^0$ of the
semigroup $\ULax$ is densely defined on $\Kscr$ and is dissipative
follows from Remark 3.4 and Theorem 1.3 in \cite{CranPazy}. On
invoking Theorem A2 in \cite{CranPazy} we obtain that $\GothA^0$ has a
maximal dissipative extension $\GothA$ (which may be set-valued) with
the same domain $\Dscr(\GothA)=\Dscr (\GothA^0)$. It follows from
Lemma 2.2 and Definition 2.3 in \cite{CranPazy} that for any
$z_0\in\Dscr(\GothA^0)$, $\GothA^0 z_0$ is the element of minimal norm
in the closed and convex set $\GothA z_0$.

The second part of the above theorem is a straightforward consequence 
of Lemma 3.1, Lemma 3.2 and Corollary 3.1 in \cite{CranPazy}. 
$\quad\square$

If $\ULax$ and $\GothA$ are related as in the above theorem, then we
say that $\ULax$ {\it is determined by} $\GothA$. By $\GothA^0$ being
densely defined in $\Kscr$ we mean that $\Dscr(\GothA^0)$ is dense in
$\Kscr$. In the linear case, $\GothA^0=\GothA$.

{\color{blue}
\begin{theorem}[Crandall-Pazy] \label{Crandall-Pazy}
Let $\GothA$ be a maximal dissipative set-valued operator on $Z$ with
domain $\Dscr(\GothA)\subset Z$. For each $z_0\in\Dscr(\GothA)$ let
$\GothA^0z_0$ denote the element of smallest norm in the set $\GothA
z_0$. Then there is a strongly continuous contractive semigroup of 
nonlinear operators $\ULax$ acting on $\overline{\Dscr(\GothA)}$ 
(which is convex) such that $\GothA^0$ is the generator of $\ULax$.
\end{theorem}}
\smallskip

The above theorem is taken from \cite{CranPazy}, see their Theorem
A1. It is a generalization of the well known
Lumer-Phillips theorem from linear semigroup theory. We refer to
\cite{Engel_Nagel, Pazy2} and \cite{obs_book} for linear semigroup
theory.
% A related result has been
%proved in \cite{CisVel} for time varying nonlinear systems in which
%the nonlinear operator is continuously Fr\'echet differentiable. 
%It is proved in \cite{CisVel} that if the operator measure of the
%Fr\'echet differential of the nonlinear operator is negative for all
%$t\geq 0$, then the system is contractive. 

%%%%%%%%%%**********%%%%%%%%%%**********%%%%%%%%%%**********%%%%%%%%%%
\section{Projected differential inclusions and equations on closed and 
         convex sets} \label{sec3} % Section 3

In this section we investigate projected differential inclusions and
differential equations on closed and convex subsets of a real Hilbert
space $Z$.

\subsection{Projections on closed convex sets}
Following \cite{brog,cojo1}, we denote the {\rm projection} of $z\in
Z$ onto the closed and convex subset $\Kscr \subset Z$ by
$P_\Kscr(z)$. For any $z\in Z$, $P_\Kscr(z)$ is a unique element in
$\Kscr$ such that
\begin{equation*} \label{projection}
   \|P_\Kscr(z)-z\| \m=\m \underset{y\in \Kscr}{\rm min}\|y-z\|.
\end{equation*}

For any $z \in \Kscr$, the {\em tangent cone} $T_\Kscr(z)$ of the set
$\Kscr$ at $z$ is the set of all those $v\in Z$ such that there exist
a sequence $(z_n)_{n \in \mathbb{N}} \in \Kscr$ and a sequence
$\delta_n \longrightarrow 0^+$ for $n\rarrow \infty$, which satisfy
$$\lim\limits_{n \to + \infty}\frac{z_n-z}{\delta_n} \m=\m v.$$
This is a general definition, valid also for non-convex $\Kscr$. For
convex $\Kscr$, $T_\Kscr(z)$ can be defined more simply as the closure
of the set of all the vectors of the form $\l(x-z)$, where $x\in\Kscr$
and $\l>0$. 

The {\em normal cone} $N_\Kscr(z)$ at a point $z\in \Kscr$ is
\begin{equation} \label{N}
   N_\Kscr(z) \m=\m \left\{ p\in Z\ \Big|\, \langle p,z-q \rangle 
   \geq 0 \ \ \ \forall q\in\Kscr \right\},
\end{equation}
and $N_\Kscr(z)= \emptyset$ for $z \notin \Kscr$. Thus, 
$\Dscr(N_\Kscr)=\Kscr$ and $N_\Kscr(z)=\{0\}$ if $z\in{\rm int}\m
\Kscr$. The following theorem is a particular case of 
\cite[Prop.~1]{moreau}.

{\color{blue} \begin{theorem}[Moreau] \label{Mor} 
With the above notation, for $z\in \Kscr$ and $x\in Z$, if we denote
the projection of $x$ onto $T_\Kscr(z)$ by $p=P_{T_\Kscr(z)}(x)$ and
the projection of $x$ onto $N_\Kscr(z)$ by $q= P_{N_\Kscr(z)}(x)$,
then $x=p+q$ and $\langle q, p\rangle =0$.
\end{theorem}} 
\smallskip

For $z,v \in Z$, let us define
\begin{equation*} \label{Pi_k}
   \Pi_\Kscr\left(z,v\right) \m=\m \left\{\begin{array}{c c}
   \underset{w \in T_\Kscr(z)}{\rm arg\m min} \|w-v\| & \mbox{if } 
   z\in \Kscr\m,\\ \emptyset & \mbox{if }z\notin\Kscr \m,\end{array}
   \right. 
\end{equation*}
where $\underset{s \in S}{\rm arg\m min}\m F(s) $ denotes the element
in $S$ that minimizes the function $F$. It is easy to see that
$\Pi_\Kscr\left(z,v \right)$ is the projection of $v$ onto the tangent
cone $T_\Kscr(z)$, i.e., 
$$\Pi_\Kscr\left(z,v\right)=P_{T_\Kscr(z)}
\left(v\right)\m.$$

By Moreau's Theorem \ref{Mor}, we obtain that for every $z \in \Kscr$
and $v\in Z$, $\Pi_{\Kscr}(z,v)$ is the element of minimal norm in the
set $v-N_\Kscr(z)$. In another words, we have
\begin{equation} \label{Copenhagen}
   \Pi_\Kscr(z,v) \m=\m \underset{q \in v-N_\Kscr(z)}{\rm arg\m min}
   \|q\|.
\end{equation}

Actually computing the projection $\Pi_\Kscr
(z,v)$ is not a trivial matter. If $Z=\rline^n$ and $\Kscr$ is 
generated by $n$ independent vectors, then there is an algorithm for 
computing $\Pi_\Kscr(z,v)$ in finitely many steps, see \cite{Nemeth}.

We extend the definition of $\Pi_\Kscr(\cdot,\cdot)$ so that the
second argument is a set $S\subset Z$. For $z \in Z$, we denote by
$\Pi_\Kscr(z,S)$ the set $\left\{\Pi_\Kscr(z,s)\big|\, s\in S
\right\}$. Thanks to equation \rfb{Copenhagen}, for $z\in\Kscr$, we
obtain that for any $S\subset Z$, \vspace{-1mm}
\begin{equation} \label{LutonTown}
   \Pi_\Kscr(z,S)=\left\{\underset{q \in s-N_\Kscr(z)}{\rm arg\m min} 
   \|q\|\Bigg|\, s\in S \right\}\m.
\end{equation}

\subsection{Projected differential inclusions on closed and convex 
sets} Let $\Kscr$ be a closed and convex subset of $Z$ such that 
${\rm int} (\Kscr)\neq \emptyset$. We will investigate the abstract 
Cauchy problem \rfb{Onana}, where $\GothA$ is a maximal dissipative 
(possibly set-valued) operator with its domain $\Dscr(\GothA)$ such 
that $\Kscr\subset\overline{\Dscr(\GothA)}$. 

\begin{definition} \label{Solution_def}
A Lipschitz continuous and right differentiable function $z:\,[0,
\infty)$ $\rarrow\Dscr(\GothA)\cap\Kscr$ which satisfies \rfb{Onana} 
for all $t\geq 0$, is called a {\em solution} of \rfb{Onana}.
\end{definition}

Note that the Lipschitz continuity of $z(\cdot)$ implies that
$z(\cdot)$ is differentiable almost everywhere. The following theorem 
states the existence of unique solutions of the projected differential
inclusion \rfb{Onana}.

\smallskip
{\color{blue} \begin{theorem} \label{th1}
Define the operator $\GothA^0_\Kscr$ as follows: \vspace{-1mm}
\begin{equation} \label{eq:2.9}
   \GothA^0_\Kscr z \m=\m \underset{w \in \Pi_{\Kscr}(z,\GothA z)} 
   {\rm arg\m min} \|w\| \m\m ,
\end{equation}
with $\Dscr(\GothA^0_\Kscr)=\Dscr(\GothA)\cap\Kscr$. Then\m:
\begin{enumerate}[(i)]
\item \label{item0} The operator $\GothA_\Kscr=\GothA-N_\Kscr$, with
      the domain $\Dscr\left(\GothA_\Kscr\right)=\Dscr\left(\GothA^0
      _\Kscr\right)$, is maximal dissipative, and for every $z\in
      \Dscr\left(\GothA_\Kscr\right)$, $\GothA^0_\Kscr z$ is the 
      element of minimal norm in the set $\GothA_\Kscr z$.
\item \label{item1} $\GothA^0_\Kscr$ is the strong generator of a 
      strongly continuous semigroup $\ULax^\Kscr=\left(\ULax^\Kscr_t
      \right)_{t\geq 0}$ of contractions on $\Kscr$.
\item \label{item2} For every $z_0 \in \Dscr(\GothA^0_\Kscr)$, the 
      function $z:[0,\infty)\rarrow Z$ defined by $z(t)=\ULax^\Kscr_t
      z_0$ is Lipschitz continuous and right differentiable,
\begin{equation*} \label{eq:3}
   z(t)\in\Dscr(\GothA^0_\Kscr),\qquad \frac{\dd^+}{\dd t}z(t) \m=\m 
   \GothA^0_\Kscr z(t) \FORALL t\geq 0,
\end{equation*}
and $\GothA^0_\Kscr z(t)$ is right continuous for all $t\geq 0$.

\item \label{item3} For every $z_0 \in \Dscr(\GothA^0_\Kscr)$, the 
      function $ z:[0,\infty)\rarrow Z$ defined by $z(t)=
      \ULax^\Kscr_t z_0$ is the unique solution of \rfb{Onana}.
\end{enumerate}
\end{theorem}} \vspace{2mm}

Notice that for any $z\in\Dscr(\GothA_\Kscr)$, we have 
$\GothA_\Kscr^0 z\in\Pi_\Kscr(z,\GothA z)\subset\GothA_\Kscr z$.

\medskip
{\it Proof:\m\m} First we prove \rfb{item0}. From the definition of
$\GothA^0_\Kscr$ (in \rfb{eq:2.9}) and from \rfb{LutonTown}, we 
obtain that for any $z\in\Dscr(\GothA^0_\Kscr)$,
\begin{equation} \label{Min_el}
   \GothA^0_\Kscr z \m=\m \underset{q \in \GothA z-N_\Kscr(z)}{\rm 
   arg\m min} \|q\|\m.
\end{equation}
It follows from the definition of $N_\Kscr$ and the discussion given
above \cite[Theorem 3]{Rocka} that $N_\Kscr$ is a maximal monotone
operator and $\Dscr(N_\Kscr)=\Kscr$. Thus, $-N_\Kscr$ is maximal
dissipative. Since the interior of $\Kscr$ is nonempty by assumption
and the closure of $\Dscr(\GothA)$ contains $\Kscr$, it follows that
$\Dscr(\GothA) \cap {\rm int} \Kscr$ is dense in ${\rm int} \Kscr$, in
particular, it is not empty. Thus, we obtain that $\Dscr(\GothA)
\cap{\rm int}( \Dscr(-N_\Kscr))\neq\emptyset$. It follows from
Rockafellar's Theorem \ref{Rockafellar} that $\GothA_\Kscr$ is maximal
dissipative. It follows from \rfb{Min_el} that for any $z_0\in\Dscr
(\GothA_\Kscr)$, $\GothA_\Kscr^0z_0$ is the element of smallest norm
in the set $\GothA_\Kscr z_0$, which completes the proof of 
\rfb{item0}. Part \rfb{item1} follows from the Crandall-Pazy theorem 
(Theorem \ref{Crandall-Pazy}).

Item~\rfb{item2} and Item~\rfb{item3} follow from
the second part of Theorem \ref{Putin}, by replacing $\GothA$ and
$\GothA^0$ in Theorem \ref{Putin} with $\GothA_\Kscr$ and
$\GothA^0_\Kscr$, respectively. For the uniqueness part, notice that
any solution $z$ of \rfb{Onana} is automatically a solution of
\vspace{-2mm}
$$ \frac{\dd^+}{\dd t}z(t)\in \GothA_\Kscr z(t) \FORALL t\geq 0,$$ 
which is unique according to Theorem \ref{Putin}. $\quad\square$

%%%%%%%%%%**********%%%%%%%%%%**********%%%%%%%%%%**********%%%%%%%%%%
\section{Differential inclusions with finitely many scalar 
   constraints} \label{sec4} % Section 4

By a {\em scalar constraint} we mean a constraint imposed on a 
continuous linear functional of the state of the system. It is worth 
examining the case when finitely many scalar constraints are imposed 
on the solution of a maximally dissipative differential inclusion, 
because such systems occur in various applications. For instance, the 
differential inclusion may describe a closed-loop system where the 
controller contains one or more saturating integrators, thus
constraining the integrator states to fixed intervals \cite{art122}. 
The examples discussed in Sections \ref{sec7} and \ref{sec8} fit this
framework. First we consider the case of a single scalar constraint.

%%%%%%%%%%**********%%%%%%%%%%**********%%%%%%%%%%**********%%%%%%%%%%
\subsection{Differential inclusions with one scalar constraint.} We 
give an application of Theorem \ref{th1} to a differential inclusion
\begin{equation} \label{single-value}
   \frac{\dd^+}{\dd t}z(t) \m\in\m \GothA z(t) \ \ \ \ \forall 
   t\geq 0\m,
\end{equation}
where $\GothA:\Dscr(\GothA)\rarrow Z$ is a maximal dissipative
(possibly nonlinear) operator on $Z$ with $\Dscr(\GothA)$ dense in
$Z$. The initial condition is $z(0)=z_0\in\Dscr(\GothA)$. It follows
from Theorems~\ref{Putin} and \ref{Crandall-Pazy} that $\GothA$
determines a semigroup of (possibly nonlinear) contraction operators,
denoted by $\left(\ULax_t\right)_{t\geq 0}$. Moreover, the unique 
solution of \rfb{single-value} is $z(t)=\ULax_tz_0$. Our aim is to 
find the correct model of a modification of this system such that a 
one-dimensional component of the state $z(t)$ is constrained to a 
closed interval, but the behavior remains unchanged if this 
one-dimensional component is in the interior of the interval. A 
simpler version of this analysis is in our conference paper 
\cite{singh2023CDC}, where the operator $\GothA$ is assumed to be 
linear.

We denote by $Z_c$ a one dimensional subspace of $Z$, such that
$Z_c=\{\l \varphi\,|\, \l\in\rline\}$, where $\varphi\in Z$ with 
$\|\varphi\|_Z=1$. On $Z$ we define an operator $P:Z\rarrow Z_c$ by
\begin{equation} \label{Proj_1}
   P z \m=\m \langle z,\varphi\rangle \varphi \FORALL z\in Z.
\end{equation}
Thus, $P$ is the orthogonal projection onto $Z_c$ and $I-P$ is 
the orthogonal projection onto $Z_c^{\perp}$. For $z\in Z$ we 
denote \vspace{-3mm}
\begin{equation*} \label{part_x}
   \m\ \ \ \ \bbm{z_1\\ z_2} \m=\m \bbm{\langle z,\varphi\rangle \\ 
   (I-P)z},
\end{equation*}
so that we have $z=z_1\varphi+z_2$. By a slight abuse of notation, 
we identify $z$ with $[z_1\ z_2]^\top$: \vspace{-2mm}
\begin{equation} \label{x_parti}
   z \m=\m \bbm{z_1\\ z_2} \ \ \ \forall z\in Z \m,\ \ z_1\in\rline,\ 
   \ z_2\in Z_c^\perp \m.
\end{equation} 
Using the above splitting of vectors in $Z$, we have that for any
$z(t)=\sbm{z_1(t)\\z_2(t)}\in \Dscr(\GothA)\subset Z$, we can rewrite
\rfb{single-value} as
\begin{equation*} \label{eq:1bis}
   \frac{\dd^+}{\dd t} \bbm{z_1(t) \\ z_2(t)} \m\in\m \GothA\bbm{ 
   z_1(t) \\ z_2(t) } \m=\m \bbm{\langle\GothA z(t),\varphi\rangle \\
   (I-P)\GothA z(t)} \FORALL t\geq 0 \m.
\end{equation*}

Let $w_{\mathrm{min}}$, $w_{\mathrm{max}}\in \rline$ such that
$w_{\mathrm{min}}<w_{\mathrm{max}}$. We want to constrain $z_1(t)$ to
the interval $[w_{\rm min},w_{\rm max}]$ for all $t\geq 0$. In other 
words, we want to constrain $z(t)$ to the closed and convex set 
$\Kscr$ for all $t\geq 0$, where
\begin{equation*}
   \begin{aligned} \Kscr \m=\m \left[w_{\rm min}, w_{\rm max} \right]
   \times Z_c^\perp \m=\m \left\{ \left. \bbm{z_1 \\ z_2} \in
   \begin{array}{c}\rline\vspace{-1mm}\\ \times\vspace{-1mm} \\ Z_c
   ^{\perp}\end{array}\n \m\m\right|\ z_1 \in \left[w_{\rm min},
   w_{\rm max}\right] \right\}. \end{aligned}
\end{equation*}
The projected differential inclusion is
\begin{equation} \label{Pibis}
   \frac{\dd^+}{\dd t} \bbm{z_1(t) \\ z_2(t)} \m\in\m \Pi_\Kscr
   \left(z(t),\GothA z(t)\right)\quad  t\geq 0\m,
\end{equation}
with the initial state $z(0)\in\Dscr(\GothA)\cap\Kscr$. For any 
$w,g\in\rline$, define the function $\Sscr$ by
\begin{equation} \label{saturation}
   \Sscr(w,g) \m= \left\{\begin{array}{lll} \text{\rm max}\{g,\m 0\}
   \hspace{0.47cm} \mbox{if}\hspace{0.5cm} w\leq w_{\rm min},\\
   g \hspace{1.2cm} \hspace{0.47cm} \mbox{if}\hspace{4.7mm} 
   w_{\rm min}<w<w_{\rm max},\\ \text{\rm min}\{g,\m 0\}\hspace{5.5mm} 
   \mbox{if}\hspace{0.5cm} w\geq w_{\rm max}. \end{array}\right.\m 
\end{equation}
Then for $z(t)\in\Dscr(\GothA)\cap\Kscr$, 
\begin{equation} \label{Proj_eqn}
   \Pi_\Kscr\left(z(t),\GothA z(t)\right) \m=\m \left\{ \left.
   \bbm{\Sscr\left(z_1(t),\langle q,\varphi\rangle\right)\\ 
   (I-P)q}\right|\ q\in\GothA z(t)\right\}. \vspace{-1mm}
\end{equation}

The function $\Sscr$ has been used extensively to constrain the state
of PI controllers in \cite{art122},\cite{Natarajan2017}. The operator
$\Pi_\Kscr$ is as stated in \rfb{Proj_eqn} and \rfb{saturation} since
the normal cone $N_\Kscr$ of $\Kscr$ is \vspace{-2mm}
$$ N_\Kscr \left(\sbm{z_1(t)\\ z_2(t)}\right) = \left\{
   \begin{array}{c c c} (-\infty ,0]\times \{0\} & \text{if} & z_1(t)
   = w_{\rm min},\\ \hspace{-0.3cm}\{0\}\times \{0\} & \text{if} & 
   \hspace{-0.2cm}w_{\rm min}<z_1(t)<w_{\rm max},\\ \left[0,\infty
   \right)\times \{0\}& \text{if}& z_1(t) = w_{\rm max}.
   \end{array}\right.$$
The above formula follows from the definition of the normal cone in
\rfb{N}. Thus, when $z\in\Dscr(\GothA)\cap\Kscr$ and $z_1=w_{\rm 
min}$, then $\GothA_\Kscr=\GothA-N_\Kscr$ (as introduced in Theorem 
\ref{th1}) is given by \vspace{-2mm}
$$ \GothA_\Kscr z \m=\m \left\{\left.\bbm{ \langle q,\varphi\rangle + 
   [0,\infty)\m \\ (I-P)q}\right|\ q\in \GothA z\right\}\m.$$
Similarly when $z_1=w_{\rm max}$, then
$$ \GothA_\Kscr z \m=\m \left\{\left.\bbm{ \langle q,\varphi\rangle - 
   [0,\infty)\m \\ (I-P)q}\right|\ q\in\GothA z\right\}\m,$$
and, when $w_{\rm min}<z_1<w_{\rm max}$, then $\GothA_\Kscr z=\GothA
z$. It follows that if $\GothA$ is single-valued, then \vspace{-2mm}
$$ \GothA_\Kscr^0 z \m=\m \bbm{{\rm max}\{\langle\GothA z,\varphi
   \rangle,\m 0\}\\ (I-P) \GothA z} \ \mbox{ for }\ \ z_1 =
   w_{\rm min} \m,$$
$$ \GothA_\Kscr^0 z \m=\m \bbm{{\rm min}\{\langle\GothA z,\varphi
   \rangle,\m 0\}\\ (I-P) \GothA z} \ \mbox{ for }\ \ z_1 =
   w_{\rm max} \m,$$
and finally, when $w_{\rm min}<z_1<w_{\rm max}$, then 
$\GothA_\Kscr^0 z=\GothA z$.

The conclusion of this subsection is that the operator $\GothA_\Kscr$
defined above is maximal dissipative and for every initial condition
$z(0)\in\Dscr(\GothA)$ that satisfies $z_1(0)\in [w_{\rm min},w_{\rm
max}]$, the differential inclusion $\frac{\dd^+} {\dd t}z(t)\in\GothA
_\Kscr z(t)$ (equivalently, \rfb{Pibis}) has a unique solution $z$ 
that is Lipschitz continuous and right differentiable for all $t\geq 
0$, and $z_1(t)\in [w_{\rm min}, w_{\rm max}]$. This follows from
\rfb{Proj_eqn} and Theorem~\ref{th1}.

%%%%%%%%%%**********%%%%%%%%%%**********%%%%%%%%%%**********%%%%%%%%%%
\subsection{Extension to $n$ scalar constraints.} \label{Marx}

One might think that the problem of modeling a system with $n$ scalar 
constraints can be understood by applying one scalar constraint, then 
another, and so on. Unfortunately, that would not work. When we 
discussed the imposition of one scalar constraint (in the previous
subsection), we have assumed that the state space of the unconstrained 
system is the entire Hilbert space $Z$. Without this assumption, the 
simple formula \rfb{Proj_eqn} for the projection $\Pi_\Kscr$ is not 
true. Indeed, if $z(t)$ would be at the edge of the state space of the 
system and also $z_1(t)$ would be at one end of the interval $[w_{\rm 
min},w_{\rm max}]$, then the projection to the tangent cone of $\Kscr$
may be much more complicated. Thus, we cannot use \rfb{Proj_eqn} to
impose (to model) a scalar constraint on an already constrained 
system.

We will show how to model a system with two scalar constraints. The 
generalization to $n$ scalar constraints is straightforward.

Consider the differential inclusion \rfb{single-value}, with the same 
assumptions as in \rfb{single-value}. Let $\varphi_1$ and $\varphi_2$
be independent vectors in $Z$, without loss of generality we may 
assume that $\|\varphi_1\|=\|\varphi_2\|=1$. Suppose that the 
numbers $w^1_{\rm min}<w^1_{\rm max}$ and $w^2_{\rm min}<w^2_{\rm 
max}$ are given. We want to model the system resulting by imposing 
here two scalar interval contraints:
\begin{equation} \label{Nasrallah}
   \langle z(t),\varphi_1 \rangle \m\in\m [w^1_{\rm min},
   w^1_{\rm max}]\m, \qquad \langle z(t),\varphi_2 \rangle \m\in\m 
   [w^2_{\rm min},w^2_{\rm max}] \m.
\end{equation}

Let $Z_c={\rm span}\m\{\varphi_1,\varphi_2\}$ and let $P$ be the 
orthogonal projection from $Z$ to $Z_c$. After choosing an 
orthonormal basis in $Z_c$ and representing any $z\in Z_c$ by its
coordinates in this basis, we can identify $Z_c$ with $\rline^2$. We
represent vectors in $Z$ in the form \rfb{x_parti}, but now with 
$z_1\in\rline^2$ and $z_2\in Z_c^\perp$. The constraints 
\rfb{Nasrallah} define a closed convex set $\Kscr_0\subset\rline^2$,
and $\Kscr=\Kscr_0\times Z_c^\perp$. Then for $z(t)\in\Dscr(\GothA)
\cap\Kscr$ we have the following generalization of \rfb{Proj_eqn}:
\begin{equation} \label{Galloway}
   \Pi_\Kscr\left(z(t),\GothA z(t)\right) \m=\m \left\{ \left.
   \bbm{\Pi_{\Kscr_0}(z_1(t),Pq)\\ (I-P)q}\right|\ q\in\GothA z(t)
   \right\}. \vspace{-1mm}
\end{equation}
The projection $\Pi_{\Kscr_0}(z_1(t),Pq)$ acting in $\rline^2$ can be
computed, for instance, using the algorithm presented in
\cite{Nemeth}. The constrained system is described by \rfb{Pibis} and
its state trajectories (assumed to start in $\Kscr$) will remain in
$\Kscr$.

%%%%%%%%%%**********%%%%%%%%%%**********%%%%%%%%%%**********%%%%%%%%%%
\section{Incrementally scattering passive nonlinear systems}
   \label{sec5} % Section 5

In this section we recall some background on well-posed (possibly
nonlinear) systems following \cite{ShWeTu:art149, ShWe:art157,
ShWe:art158}. However, we slightly generalize the framework of
\cite{ShWeTu:art149, ShWe:art157, ShWe:art158} by allowing systems
whose state space is not a Hilbert space, but a closed and convex
subset of a Hilbert space. We denote by $U$ the {\it input space}, by
$K$ the {\it state space} and by $Y$ the {\it output space} of a
well-posed (possibly nonlinear) system \m $\Sigma$. Here $U$ and $Y$
are real Hilbert spaces while $K$ is a closed and convex subset of a
real Hilbert space $X$. The input and the output functions are $u\in
L^2_{\rm loc}([0,\infty); U)$ and $y\in L_{\rm loc}^2([0,\infty); Y)$,
respectively. For any $y\in L^2_{\rm loc}([0,\infty);Y)$ and any
$\tau\geq 0$, we denote by $\PPP_\tau y$ the truncation of $y$ to the
interval $[0,\tau]$. $\PPP_\tau y$ is in $L^2 ([0,\infty);Y)$ and it
is zero for $t>\tau$. We denote $\Uscr=L^2([0,\infty);U)$ and
$\Yscr=L^2((-\infty,0];Y)$. The space $\Yscr$ can be considered as a
subspace of $L^2((-\infty, \infty);Y)$ by extending any function in
$\Yscr$ to be zero for $t>0$.

\begin{definition} \label{Shejaiya_tragedy}
A {\it time invariant well-posed} (possibly nonlinear) system $\Sigma$
on the input space $U$, the state space $K$ and the output space $Y$,
consists of two families of (possibly nonlinear) continuous operators
\vspace{-1mm}
$$ \Sigma^{\rm  st} \m=\m (\Sigma^{\rm  st}_t)_{t\geq 0} \m,\quad
   \Sigma^{\rm out} \m=\m (\Sigma^{\rm out}_t)_{t\geq 0} \m,
   \vspace{-1mm}$$
where $\Sigma^{\rm st}_t:K\times\Uscr\rarrow K$ and $\Sigma^{\rm
out}_t:K\times\Uscr\rarrow L^2([0,t];Y)$ are such that the following
family of operators $\ULax=(\ULax_t)_{t\geq 0}$ is a strongly
continuous semigroup of (possibly nonlinear) operators acting on
$\Yscr\times K\times\Uscr$: for every $t\geq 0$,
\begin{equation} \label{J_Pasha}
   \ULax_t \m=\m \bbm{\Sscr_{-t} & 0 \sbluff & 0 \\
   0 & I & 0 \\ 0 & 0 & \SSS_t^*} \left[\begin{array}{c|c c} I & \ 
   \Sigma^{\rm out}_t \\ 0 & \ \Sigma^{\rm st} _t \\ \hline
   \ \ 0 \ \  & 0\ \ \ I \end{array}\right]\m,
\end{equation}
where $\SSS_t$ is the (unilateral) right shift operator by $t$ on
$\Uscr$ and $\Sscr_t$ is the bilateral right shift by $t$ acting on
$L^2((-\infty,\infty);Y)$. Thus, the adjoint operator $\SSS^*_t$ is
the operator of left shift by $t$ on $\Uscr$ and $\Sscr_{-t}$ is
the operator of left shift by $t$ on $L^2((-\infty,\infty);Y)$ (thus,
also on $\Yscr$).

Moreover, we require that for all $\tau\geq 0$, $v\in\Uscr$ and 
$x_0\in X$, 
\begin{eqnarray}\label{Identities}
   \Sigma^{\rm st}_\tau\bbm{x_0\\ u} = \Sigma^{\rm st}_\tau
   \bbm{x_0\\ \PPP_\tau u},\ \ \Sigma^{\rm out}_\tau\bbm{x_0\\ u}
   = \Sigma^{\rm out}_\tau\bbm{x_0\\ \PPP_\tau u},
\end{eqnarray}
for all $u\in L^2([0,\infty);U)$ and $x(0)\in K$. 
\end{definition}

The identities \rfb{Identities} are called the {\em causality
conditions}. 

\begin{remark}\label{Doku}
For $\Sigma$ as above, the state and the output function are described
as follows\m: For all $\tau\geq 0$, \vspace{-2mm}
\begin{equation} \label{Lea_Pais}
   x(\tau) \m=\m \Sigma^{\rm st}_\tau \bbm{x_0\\ \PPP_\tau u} \m,
   \qquad \PPP_\tau y \m=\m \Sigma^{\rm out}_\tau \bbm{x_0\\ 
   \PPP_\tau u} \m.\vspace{1mm}
\end{equation}
\end{remark}
A related definition has appeared in \cite{Mironchenko}. 

A nonlinear system $\Sigma$ is {\it Lipschitz continuous} if
$\Sigma_\tau^{\rm st}$ and $\Sigma_\tau^{\rm out}$ are Lipschitz
continuous (with Lipschitz constants that may depend on $\tau$). In
particular, $\Sigma$ is called {\em incrementally scattering passive}
if the following is true: If $x_{01}$, $x_{02}$ are initial states in
$K$, $u_1$, $u_2$ are input functions in $ \Uscr$, $x_1$, $x_2$ are
the corresponding state trajectories of the nonlinear system $\Sigma$,
as in \rfb{Lea_Pais}, and $y_1$, $y_2$ are the corresponding output
functions of $\Sigma$, as in \rfb{Lea_Pais}, then for all $\tau\geq
0$, the following {\it energy balance inequality} is satisfied\m:
\vspace{-1mm}
\begin{equation} \label{energy_balance_M} 
  \|x_1(\tau)-x_2(\tau)\|^2 + \hspace{-1mm}\int_0^\tau \|y_1(t)-
  y_2(t)\|^2 \dd t  \m\leq\m \|x_{01}-x_{02}\|^2 + \hspace{-1mm}
  \int_0^\tau \|u_1(t)-u_2(t)\|^2 \dd t \m.
\end{equation}
In other words, for any $\tau\geq 0$, the operator $\Sigma_\tau=
\bbm{\Sigma^{\rm st}_\tau\\ \Sigma^{\rm out}_\tau}$ is a contraction. 

Following \cite{ShWe:art158} every incrementally passive system has
the following ``local in time"  representation:

{\color{blue} \begin{theorem} \label{incre_passive} 
Using the notation of this section, consider an incrementally
scattering passive well-posed system $\Sigma$, described by two
families of operators as in \rfb{Lea_Pais}. Then there exist (possibly
nonlinear) operators $[\AB]$ (which may be multi-valued) and
$[\CD]$ (single-valued), 
$$ [\AB]:\Dscr([\AB]) \rarrow X \m,\qquad [\CD]:\Dscr([\AB])\rarrow 
   Y\m,$$ 
where $\Dscr([\AB])$ is dense in $K\times U$, such that the following 
holds:

If $u\in\Hscr^1((0,\infty);U)$ and $[x(0)\ u(0)]^\top\in\Dscr
([\AB])$, then the corresponding state trajectory $x$ and the output 
function $y$ satisfy $[x(t)\ u(t)]^\top \in\Dscr([\AB])$ for all 
$t\geq 0$ and \vspace{-2mm}
\begin{equation} \label{Ax+Bu}
  \left\{\begin{array}{ll}\frac{\dd^+}{\dd t} x(t) \m\in \m [\AB]
  \bbm{x(t)\\u(t)} \FORALL t\geq 0,\vspace{2mm}\\ y(t) \m=\m 
  [\CD]\bbm{x(t)\\ u(t)} \ \ \ \ \FORALL t\geq 0. \end{array}\right.
\end{equation} 
\end{theorem}}
\smallskip

In \cite{ShWe:art158}, this theorem was stated (and proved) for $K=X$,
but it remains true as stated here (i.e., a little more general), with
minor adaptations of the proof. Specifically, in Step 1 of the proof
of \cite[Theorem 6.1]{ShWe:art158}, the contraction Lax-Phillips
semigroup denoted by $\ULax^{\rm NL}$ (in \cite{ShWe:art158}) should
be replaced by $\ULax$ (given by \rfb{J_Pasha}), defined on the
product space $\Yscr\times K\times \Uscr$. Consequently, the domain
$\Dscr(\GothA)$ of the maximal dissipative operator $\GothA$ that
determines $\ULax$, will be dense in the product space $\Yscr\times
K\times\Uscr$. The remaining steps in the proof of \cite[Theorem
6.1]{ShWe:art158} are similar, only replacing $X$ with $K$.

It is easy to see that for an incrementally scattering passive system
$\Sigma$, the semigroup $\ULax$ from \rfb{J_Pasha} is a contraction
semigroup of (possibly nonlinear) operators acting on the product
space $\Yscr\times K\times \Uscr$.

If we take $u_0\in\Uscr$, $x_0\in K$ and $y_0\in\Yscr$ to represent
the future input function of $\Sigma$, its initial state and its past
output function (for negative time), respectively, then at any time
$t\geq 0$,
\begin{equation} \label{SemigroupT}
   \bbm{y_t\\ x(t)\\ u_t} \m=\m \ULax_t \bbm{y_0\\ x_0\\ u_0}. 
\end{equation}
Here $y_t$ represents the past output function up to time $t$, $x(t)$
is the state at time $t$, and $u_t$ represents the future input 
function that will reach the system after time $t$.

The contraction semigroup $\ULax$ introduced above in \rfb{J_Pasha}
is called the {\em Lax--Phillips semigroup} of $\Sigma$. We denote 
the generator of $\ULax$ by $\GothA^0$. It follows from Theorem 
\ref{Putin} that $\GothA^0$ is dissipative and its domain 
$$ \Dscr(\GothA^0) = \left\{\left.\nm\bbm{y_0\vspace{2mm}\\ x_0
   \vspace{2mm}\\ u_0}\nm\in\hspace{-4mm}\begin{array}{c}\Hscr^1
   ((-\infty,0);Y)\\ \times \\ K\\ \times\\ \Hscr^1((0,\infty);U)
   \end{array}\hspace{-2mm}\right|\hspace{-2mm}\begin{array}{c} 
   \bbm{x_0\\ u_0(0)}\in\Dscr(\AB),\\ y_0(0)=[\CD]\bbm{x_0\\ u_0(0)}
   \end{array}\hspace{-2mm}\right\}$$
is dense in the product space $\Yscr\times K\times\Uscr$. Moreover,
$\GothA^0$ has a maximally dissipative extension $\GothA$ which may
be set-valued and has the same domain, i.e., $\Dscr(\GothA)=\Dscr
(\GothA^0)$. For any $[y_t\ x(t)\ u_t]^\top\in\Dscr(\GothA^0)$, the
differential inclusion governing the evolution of $[y_t\ x(t)\ u_t]
^\top$ is \vspace{-2mm}
\begin{equation} \label{GothAact}
   \frac{\dd^+}{\dd t} \bbm{y_t\\ x(t)\\ u_t} \m\in\m \GothA 
   \bbm{y_t \\ x(t) \\ u_t} \m=\m \bbm{ \vspace{1mm} y_t' \\ 
   \vspace{1mm} [\AB] \bbm{x(t)\\ u_t(0)} \\ u_t'},
\end{equation}
where $y_t'$ and $u_t'$ denote the derivatives of $y_t$ and $u_t$.

%\begin{remark}
%In Definition \ref{Shejaiya_tragedy} and also in the definition of an
%incrementally scattering passive system, we could have omitted the
%requirement that $K$ is convex. However, if $\Sigma$ is an
%incrementally scattering passive system whose state space $K$ is a
%closed (possibly not convex) subset of $X$, and if its Lax-Phillips
%semigroup is generated (see Definition 3.2 in \cite{CranPazy}), then 
%this system can be embedded in a system with a larger, closed and 
%convex state space $\tilde{K}\subset X$. This is because the generator
%$\GothA$ of the Lax-Phillips semigroup of \m $\Sigma$ is dissipative 
%(see Remark 1.2 in \cite{CranPazy}), it determines the semigroup, and
%every dissipative operator can be extended to a maximal disspative 
%one. This extension is possible according to Zorn's lemma, and it may 
%be not unique. The closure of domain of the generator of the extension
%is convex (see Corollary 2.2 in \cite{CranPazy}), implying that the
%state space of the extended system is convex. \end{remark}

%%%%%%%%%%**********%%%%%%%%%%**********%%%%%%%%%%**********%%%%%%%%%%
\section{Projected dynamical nonlinear systems associated with an 
   incrementally scattering passive system} \label{sec6} % Section 6

Consider an incrementally scattering system $\Sigma$ on the input
space $U$, the state space $X$ and the output space $Y$ such that at
any time $t\geq 0$, the state of $\Sigma$ evolves according to the
differential inclusion in \rfb{Ax+Bu} and the output of $\Sigma$ is
described by the second row of \rfb{Ax+Bu}. Suppose that due to some
physical constraint, we must modify the system in such a way that in
the modified system $\Sigma^K$, the state is restricted to a closed
and convex subset $K$ of the state space $X$, i.e., $x(t)\in K$, for
any $t\geq 0$. For technical reasons, we assume that $K$ has a
nonempty interior\m: ${\rm int}(K)\neq\emptyset$. When $x(t)\in {\rm
int}(K)$, then the modified system should behave like the original
one. Using the projection of $\frac{\dd^+}{\dd t}x(t)$ onto the
tangent cone $T_K(x(t))$, the state of $\Sigma^K$ evolves according to
\begin{equation*} \label{Sig_k}
   \frac{\dd^+}{\dd t}x(t) \m\in\m \Pi_K\left(x(t),[\AB]\sbm{x(t)\\
   u(t)} \right),
\end{equation*}
where $[\AB]$ is as in Theorem \ref{incre_passive}, with
$\Dscr([\AB])$ dense in $X\times U$. Since $[\AB]$ is possibly
set-valued in $X$, the projection $\Pi_K$ is given by \rfb{LutonTown}
with $S=[\AB]\sbm{x(t)\\u(t)}$. The output of $\Sigma^K$ is described
by the output equation in \rfb{Ax+Bu}.

On the space $K\times U$, we denote by $[\AB]^K$ the following
operator\m: For every $t\geq 0$ and $[x(t) \ u(t)]\in \Dscr([\AB]^K)$,
\begin{equation} \label{AB^K}
   [\AB]^K\bbm{x(t)\\ u(t)} \m=\m \Pi_K\left(x(t),[\AB]\sbm{x(t)\\
   u(t)} \right)\m,
\end{equation}
with the domain $\Dscr([\AB]^K)=\Dscr([\AB])\cap (K\times U)$.  

\smallskip
{\color{blue} \begin{theorem} \label{pro_Lax} 
Let $\Sigma$ be an incrementally scattering passive system with input
space $U$, state space $\CCC$ and output space $Y$. Here $\CCC$ is a 
closed and convex subset of the real Hilbert space $X$. Let $K$ be a 
closed and convex subset of $\CCC$, such that ${\rm int}(K)$ is 
nonempty. 

Then there exists a well-posed system $\Sigma^K$ with input space $U$,
state space $K$ and output space $Y$, defined by the following
differential inclusion and output equation\m: 
\begin{equation} \label{Ax+Bu_K}
   \left\{\begin{array}{ll}\frac{\dd^+}{\dd t}x(t) \m\in\m [\AB]^K
   \bbm{x(t)\\u(t)}\m,\vspace{2mm}\\ y(t) \m=\m [\CD]\bbm{x(t)\\
   u(t)}\m, \end{array}\right. 
\end{equation} 
where $[\AB]^K$ is as in \rfb{AB^K}. If $u\in\Hscr^1((0,\infty);U)$
and the initial state $x(0)$ satisfies $[x(0)\ u(0)]^\top\in\Dscr
([\AB]^K)$, then the differential inclusion in \rfb{Ax+Bu_K} has a
unique solution $x$, meaning that the map $t\mapsto x(t)$ satisfies
$$ \hspace{2cm}\bbm{x(t)\\ u(t)} \m\in\m \Dscr([\AB]^K) 
   \FORALL t\geq 0 \m,$$
$x$ is right differentiable and Lipschitz continuous and for all
$t\geq 0$, the differential inclusion in \rfb{Ax+Bu_K} holds. For the
same $u$ and $x$, the corresponding output function $y$ of \m 
$\Sigma^K$ is given by the second part of \rfb{Ax+Bu_K}.

Moreover, $\Sigma^K$ is incrementally scattering passive, i.e., the 
solutions $x_1,x_2$ and the output functions $y_1,y_2$ corresponding 
to initial states $x_{01},x_{02}\in K$ and input functions $u_1,u_2
\in\Hscr^1((0,\infty);U)$, respectively, satisfy 
\rfb{energy_balance_M} for all $\tau\geq 0$.
\end{theorem}} \vspace{2mm}

{\it Proof\m:} We denote by $\widetilde{K}$ the closed and convex set
$\Yscr\times K\times \Uscr$. Clearly, $\widetilde{K}$ has nonempty
interior. Since $\Sigma$ is an incrementally scattering passive
system, it follows from what we said after Theorem \ref{incre_passive}
that $\Sigma$ can be represented entirely by its Lax-Phillips
semigroup $\left(\ULax_t\right)_{t\geq 0}$ as in \rfb{J_Pasha}.

The operator $\GothA$ (as in \rfb{GothAact}) determines the
Lax-Phillips semigroup $\left(\ULax_t\right)_{t\geq 0}$ (given in
\rfb{J_Pasha} and \rfb{SemigroupT}) associated with the incrementally
passive system $\Sigma$. Using Moreau's Theorem \ref{Mor}, we obtain
the that differential inclusion obtained using the projection of the
set $\GothA[y_t\ x(t)\ u_t]^\top$ of velocity vectors onto the 
tangent cone $T_{\widetilde{K}}([y_t\ x(t)\ u_t]^\top)$ is as 
follows\m: For all $[y_t \ x(t)\ u_t]^\top\in\Dscr(\GothA)\cap
\widetilde{K}$,
\begin{equation} \label{Rudy}
   \frac{\dd^+}{\dd t}\bbm{y_t\\ x(t)\\ u_t}\in\Pi_{\widetilde{K}}
   \left(\bbm{y_t\\ x(t)\\ u_t},\GothA\bbm{y_t\\ x(t)\\ u_t}\right)\m,
\end{equation}
and $\overline{\Dscr(\GothA)\cap\widetilde{K}}=\widetilde{K}$. Since 
the constraint is applied only on the state $x(t)$ of $\Sigma$, it 
follows from \rfb{GothAact} that 
$$ \Pi_{\widetilde{K}}\left(\bbm{y_t\\ x(t)\\ u_t},\GothA\bbm{y_t\\ 
   x(t)\\ u_t}\right) \m=\m \bbm{y'_t\\ \Pi_K\left(x(t),[\AB]
   \sbm{x(t)\\ u_t(0)}\right)\\ u_t'} \m.$$ 
Denote by $\GothA^0_K$ the following operator\m:
$$ \GothA^0_K z \m=\m \underset{w\in\Pi
   _{\widetilde{K}}(z,\GothA z)}{\rm arg\m min} \|w\|\m,$$
$$ \Dscr(\GothA_K^0) = \left\{\left.\nm\bbm{y_0\vspace{2mm}\\ x_0
   \vspace{2mm}\\ u_0}\in\hspace{-4mm}\begin{array}{c}\Hscr^1
   ((-\infty,0);Y)\\ \times \\ K\\ \times\\ \Hscr^1((0,\infty);U)
   \end{array}\hspace{-2mm}\right|\hspace{-2mm} \begin{array}{c} 
   \bbm{x_0\\ u_0(0)}\in\Dscr(\AB),\\ y_0(0)=[\CD]\bbm{x_0\\ u_0(0)}
   \end{array}\hspace{-2mm}\right\}\m.$$
From Theorem \ref{th1} we obtain that $\GothA^0_K$ is the generator of
a contraction semigroup $\left(\ULax_t^K\right)_{t\geq 0}$
on $\widetilde{K}$. It follows from Item
\rfb{item2} of Theorem \ref{th1} that the solution $[y_t \ x(t) \
u_t]^\top\in\Dscr(\GothA_K^0)$ of \rfb{Rudy} is Lipschitz continuous
and right differentiable for every $t\geq 0$. Therefore, the map $t
\mapsto x(t)\in K$ is Lipschitz continuous and right differentiable
for every $t\geq 0$. 

We claim that $\left(\ULax_t^K\right)_{t\geq 0}$ is a Lax-Phillips
semigroup, as in \rfb{J_Pasha}. The proof of this claim is as in the
second part of the proof of Theorem 6.3 in \cite{ShWeTu:art149},
on replacing the state space $X$ in \cite{ShWeTu:art149} with $K$, 
and the operator $[A-\Mscr\ B]$ in \cite{ShWeTu:art149} with 
$[\AB]^K$. The (possibly set-valued) monotone operator $\Mscr$ in 
\cite{ShWeTu:art149} corresponds here to the normal cone $N_K$, with 
the domain $K$.

If we denote the initial condition by $[y \ x(0)\ u]^\top\in 
\Dscr(\GothA_K^0)$, then it follow from the structure of $\ULax_t$ (as
in \rfb{J_Pasha}) that $u_t=\SSS^*_tu$ and $u(t)=u_t(0)$. From
\rfb{AB^K}-\rfb{Rudy}, on substituting $u(t)=u_t(0)$, we obtain from 
Item \rfb{item3} of Theorem \ref{th1} that the map $t\mapsto x(t)\in 
K$ is the unique solution of the differential inclusion in
\rfb{Ax+Bu_K}. 

From Item \rfb{item2} of Theorem \ref{th1}, we have that $[y_t\ x(t)\
u_t]^\top\in\Dscr(\GothA_K^0)$ for all $t\geq 0$, thus, it follows 
from the description of the domain $\Dscr(\GothA_K^0)$ (given above), 
that the output $y(t)$ at every $t\geq 0$ is given by \vspace{-2mm}
\begin{equation} \label{Newcastle}
   y(t) \m=\m [\CD]\bbm{x(t)\\ u_t(0)}\m.\vspace{-2mm}
\end{equation}
On substituting $u(t)=u_t(0)$ in \rfb{Newcastle}, we obtain the output
equation in \rfb{Ax+Bu_K}. Thus, we have proved the first part of
Theorem \ref{pro_Lax} (up to ``moreover'').

Now we prove the incremental passivity of $\Sigma^K$. Since
$\left(\ULax^K_t\right)_{t\geq 0}$ is a contraction semigroup, it
follows that for any $\tau\geq 0$ and initial conditions $[y_{01}\
x_{01}\ u_{01}]^\top$, $[y_{02}\ x_{02}\ u_{02}]^\top\in 
\widetilde{K}$, \vspace{-3mm}
\begin{equation} \label{inequal}
   \m\hspace{9mm} \left\|\bbm{y_{\tau 1}\\ x_1(\tau)\\ u_{\tau 1}} - 
   \bbm{y_{\tau 2}\\ x_2(\tau)\\ u_{\tau 2}}\right\| \m\leq\m \left\|
   \bbm{y_{01}\\ x_{01}\\ u_{01}}-\bbm{y_{02}\\ x_{02}\\ u_{02}}
   \right\| \m,\vspace{-1mm}
\end{equation}
where 
$$ \bbm{y_{\tau 1}\\x_1(\tau)\\u_{\tau 1}} \m=\m \ULax^K_\tau
   \bbm{y_{01}\\ x_{01}\\ u_{01}}, \quad \bbm{y_{\tau 2}\\ x_2(\tau)\\
   u_{\tau 2}} \m=\m \ULax^K_\tau\bbm{y_{02}\\ x_{02}\\ u_{02}}.$$
On squaring both sides of \rfb{inequal} and expanding the norm on 
$\Yscr\times X\times \Uscr$, we get \vspace{-1mm}
$$ \int_0^\tau \|y_{\tau 1}(t)-y_{\tau 2}(t)\|_Y^2 \dd t + \|x_1(\tau)
   - x_2(\tau)\|_X^2 + \int_\tau^\infty\n \|u_1(t)-u_2(t)\|_U^2\dd t  
   \vspace{-1mm}$$
$$ \hspace{5cm}\m\leq\m \|x_{01}-x_{02}\|_X^2 + \int_0^\infty\n 
   \|u_1(t)-u_2(t)\|_U^2\dd t\m.$$
On rearranging the terms above we obtain
$$ \hspace{-5cm}\int_0^\tau \|y_{\tau 1}(t)-y_{\tau 2}(t)\|_Y^2\dd t 
   + \|x_1(\tau) - x_2(\tau)\|_X^2\vspace{-1mm}$$
$$\hspace{5cm}\leq\m \|x_{01}-x_{02}\|_X^2 + \int_0^\tau
   \|u_1(t)-u_2(t)\|_U^2\dd t \m. \vspace{-1mm}$$
Therefore, $\Sigma^K$ is incrementally scattering passive. 
$\quad \square$  

Since $\Sigma^K$ is well-posed, it follows from \rfb{J_Pasha} that its
Lax-Phillips semigroup $\left(\ULax_t^K\right)_{t\geq 0}$ has the 
following structure:
\begin{equation*}
   \ULax_t^K \m=\m \bbm{\Sscr_{-t} & 0 \sbluff & 0 \\ 0 & I & 0 \\ 
   0 & 0 & \SSS_t^*} \left[\begin{array}{c|c c} I & \ \Sigma^{{\rm 
   out},K}_t \\ 0 & \ \Sigma^{{\rm st},K}_t \\ \hline \ \ 0 \ \ & 0\ 
   \ \ I \end{array}\right],
\end{equation*}
where $\Sigma^{{\rm st},K}_t:K\times \Uscr\rarrow K$ and $\Sigma_t
^{{\rm out},K}:K\times \Uscr\rarrow L^2([0,t];Y)$ are families of 
contraction operators. Therefore, for any initial state $x_0\in K$
and any $u\in\Uscr$, the state and the output function of \m 
$\Sigma^K$ are given by \rfb{Lea_Pais}, with $\Sigma^{{\rm st},K}
_\tau$ in place of $\Sigma^{\rm st}_\tau$ and $\Sigma^{{\rm out},K}_t$
in place of \m $\Sigma^{\rm out}_\tau$.

%%%%%%%%%%**********%%%%%%%%%%**********%%%%%%%%%%**********%%%%%%%%%%
\section{String equation with restricted displacement} \label{sec7}
% Section 7

Here we give an application of our main result (Theorem \ref{pro_Lax})
to the string equation on the interval $J=[0,\pi]$. The boundary input 
is applied at the left end ($\xi=0$), while the other end ($\xi=\pi$) 
is fixed. We consider that the vertical displacement of the string in 
two interior points is constrained, see Figure \ref{wave_fig}. This 
example is a generalization of an example in 
\cite[Sect.~4]{singh2023CDC}.

%%%%%%%%%%**********%%%%%%%%%%**********%%%%%%%%%%**********%%%%%%%%%%
\begin{figure}[h!] % Figure 2
   \centering 
   \includegraphics[height=3cm, width=9cm]{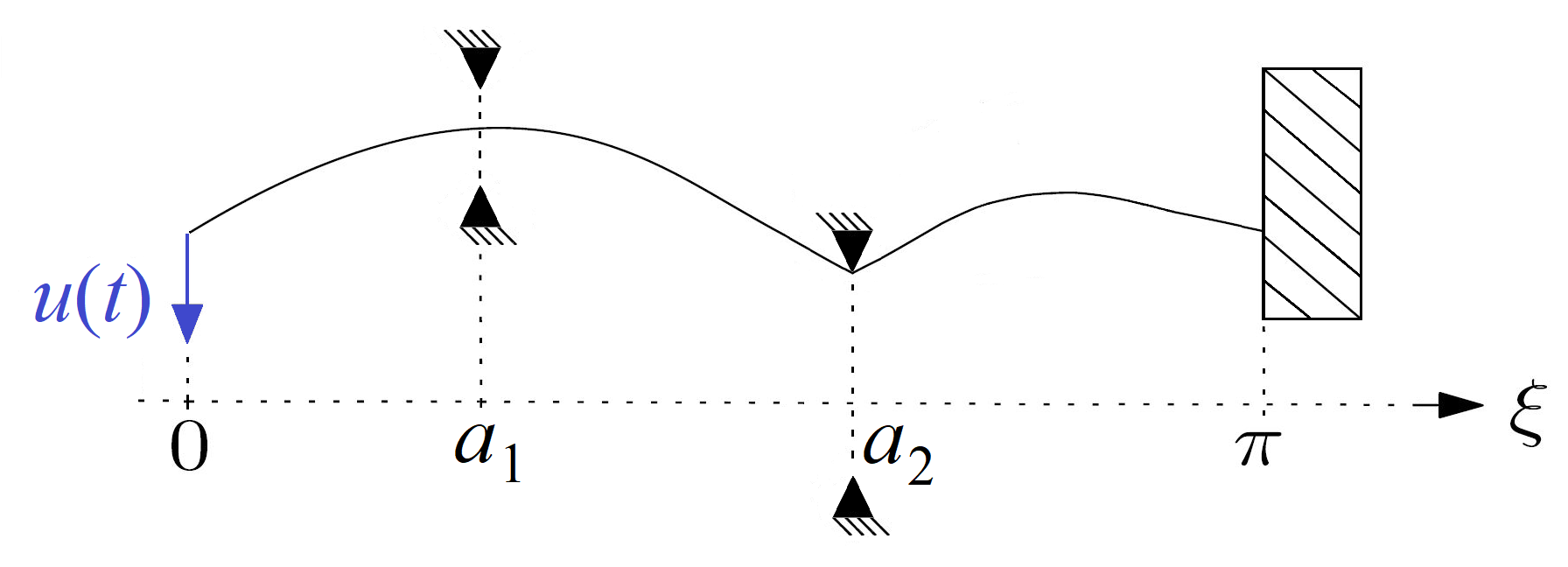}
   \caption{Vibrating string with obstacles that limit its vertical
   displacement in two points ($\xi=a_1$ and $\xi=a_2$). In the 
   figure, one constraint at $\xi=a_2$ is active.} \label{wave_fig}
\end{figure} 
%%%%%%%%%%**********%%%%%%%%%%**********%%%%%%%%%%**********%%%%%%%%%%

The partial differential and boundary equations representing the
vibrating string system $\Sigma$ are
\begin{equation}\label{wave}
   \left\{ \begin{array}{ll} \frac{\partial^2}{\partial t^2}w(\xi,t) 
   \m=\m \frac{\partial^2}{\partial \xi^2} w(x,t)\m, \quad w(\pi,t)
   \m=\m 0\m , \vspace{2mm}\\ w(\xi,0) \m=\m w_0(\xi),\ \ 
   \frac{\partial }{\partial t}w(\xi,0) \m=\m z_0(\xi)\m , 
   \vspace{2mm} \\ -\frac{\partial }{\partial \xi} w(0,t) + b^2
   \frac{\partial}{\partial t}w(0,t) \m=\m \sqrt{2}bu(t), 
   \vspace{2mm} \\ -\frac{\partial }{\partial \xi} w(0,t) - b^2
   \frac{\partial}{\partial t}w(0,t) \m=\m \sqrt{2} b y(t), 
   \end{array} \right.
\end{equation}
where $\xi\in J$ and $t\geq 0$. Here $b\neq 0$ is a constant. The 
functions $w_0$ and $z_0$ are the initial state of the system. We 
denote by $x(t)=\sbm{w(\cdot,t)\\ \frac{\partial}{\partial t}
w(\cdot,t)}$ the state of the above linear system $\Sigma$. Denote
$$\Hscr^1_r(0,\pi)=\{w\in \Hscr^1(0,\pi)\m\m|\m\m w(\pi)=0\}.$$
Then the natural state space is $X=\Hscr^1_r(0,\pi)\times L^2[0,\pi]$,
and on $X$ we define the norm as follow: For $\sbm{f\\ g}\in X$
\begin{equation} \label{normA}
   \left\| \bbm{f\\ g}\right\|^2_X \m=\m \int^{\pi}_0 \left|
   \frac{\dd f}{\dd \xi}(\xi)\right|^2 \dd \xi + \int^{\pi}_0 
   |g(\xi)|^2\dd \xi \m.
\end{equation}  
On $X$ we can define the operator $A$ as follows:
\begin{equation} \label{operatorA}
   A\bbm{f\\ g} \m=\m \bbm{0 & I\\ \frac{\dd^2}{\dd \xi^2 } & 0}
   \bbm{f\\ g}\quad \FORALL \bbm{f\\ g}\in\Dscr(A)\m,
\end{equation}
$$ \Dscr(A) \m=\m \left\{ \bbm{f \\ g}\in \hspace{-2mm}\left. 
   \begin{array}{c} \Hscr^2(0,\pi)\cap \Hscr^1_r(0,\pi)\vspace{-1mm} 
   \\ \times\vspace{-1mm} \\ \Hscr^1_r(0,\pi)\end{array} \right| 
   \frac{\dd f}{\dd \xi}(0)=b^2g(0) \right\}.$$
Clearly, $\Dscr(A)$ is dense in $X$, with the norm defined by
\rfb{normA}.  The control operator $B=[0 \ \sqrt{2}b\delta_0]^\top$,
where $\delta_0$ is Dirac delta operator at $0$, see 
\cite{singh2023CDC}.

The output equation of the vibrating string system $\Sigma$ is 
\begin{equation} \label{output_wave}
   y(t) \m=\m \bar{C}x(t)+Du(t) \m=\m -\frac{b}{\sqrt{2}}
   \frac{\partial w}{\partial t}(0,t)-\frac{1}{\sqrt{2}b}
   \frac{\partial w}{\partial \xi}(0,t),
\end{equation}
where $\bar{C}=\bbm{0 & -\sqrt{2}b\delta_0^*}$ and $D=I$. Here 
$\delta_0^*$ is the operator of point evaluation at $0$. With the 
input and the output space $U=Y=\rline$, the above system \rfb{wave}
is a linear scattering passive system, see 
\cite[Sect.~4]{singh2023CDC} (based on \cite[Sect.~7]{art54}). 

Let $a_1\in(0,\pi)$. Consider the continuous linear functional $E_1$ 
on the space $X$, defined by the point evaluation of the first 
function in $\sbm{f & g}^\top$ at the point $a_1$, that is to say:
$$E_1\bbm{f\\ g}=\m f(a_1) \m \FORALL \bbm{f\\ g}\in X.$$
By the Riesz representation theorem, there exists a unique function 
$\psi_1\in\Hscr^1_r(0,\pi)$ such that: \vspace{-3mm} 
$$ E_1 \bbm{f\\ g} \m=\m \left\langle \bbm{f\\ g}, \bbm{\psi_1\\ 0}
   \right\rangle \FORALL \bbm{f\\ g}\in X \m.$$
Denote $\bar\psi_1=\sbm{\psi_1\\ 0}$. It is easy to find that 
\vspace{-1mm}
$$ \psi_1(\xi) \m=\m \left\{\begin{array}{c c} \pi - a_1 & \mbox{for }
   \ \xi\in [0,a_1),\\ \pi-\xi  & \mbox{for }\ \xi\in [a_1,\pi] \m.
   \end{array} \right.$$
For some value $a_2\in(0,\pi)$ with $a_2\not=a_1$, we redo this 
construction with $a_2$ in place of $a_1$, obtaining the function 
$\psi_2\in\Hscr^1_r(0,\pi)$ and the vector $\bar\psi_2\in X$. It is 
easy to see that $\psi_1$ and $\psi_2$ are linearly independent.
Introduce
$$ \varphi_1 \m=\m \frac{\bar{\psi_1}}{\|\psi_1\|} \m,\qquad
   \varphi_2 \m=\m \frac{\bar{\psi_2}}{\|\psi_2\|} \m,$$
and $X_c={\rm span}\{\varphi_1,\varphi_2\}$. As in Subsection 
\ref{Marx}, we can identify $X_c$ with $\rline^2$ and then we 
have the representation $X=\rline^2\times X_c^\perp$.

The two scalar constraints that we impose on this system can be 
described using the normalized state vectors $\varphi_1,\varphi_2$ 
exactly as in \rfb{Nasrallah}, only writing $x$ in place of $z$, and 
they define a closed and convex subset $K\subset X$, which has the 
structure $K=K_c\times X_c^\perp$, where $K_c\subset\rline^2$ is 
closed and convex (a parallelogram with nonempty interior). According 
to our main result (Theorem \ref{pro_Lax}) the constrained system 
described by \rfb{Ax+Bu_K} is incrementally scattering passive 
(and it has also some other properties listed in Theorem 
\ref{pro_Lax}). The projection operator $\Pi_K$ needed to define 
$[\AB]^K$ as in \rfb{AB^K} can be expressed via a projection operator 
acting in $\rline^2$, by the technique explained in Subsection 
\ref{Marx}.

%%%%%%%%%%**********%%%%%%%%%%**********%%%%%%%%%%**********%%%%%%%%%%
\section{Projected dynamical system associated to Maxwell's 
    equations\\ on a cylinder} \label{sec8} % Section 8

As another application of our main results, we investigate the
electromagnetic field in a very much simplified representation of a
liquid metal fault current limiter (FCL). For background on FCLs we 
refer to \cite{Ganev_2012,Smedley}. We look at a relatively new type 
of FCL, called a liquid metal FCL, see \cite{Yang}. Our analysis 
refers only to the main current limiting element in a liquid metal 
FCL, which is based on the self-pinching mechanism in liquid metals
when conducting a large current, as described in \cite{He_2017}. The
pinching effect takes place in a channel represented by a bounded 
cylinder $\Om_c$ of length $L>0$ and radius $R_c$, as shown in 
Fig.~\ref{fig:domain}, filled with liquid metal. This domain is 
included in a larger cylindrical domain $\Om\subset\rline^3$, of the
same length and with the same axis. Intuitively, the region $\Om
\setminus \Om_c$ is the insulation protecting the liquid metal channel
and it inhibits the current to flow through it. We denote the radius 
of the larger cylinder $\Om$ by $R$.

%%%%%%%%%%**********%%%%%%%%%%**********%%%%%%%%%%**********%%%%%%%%%%
\begin{figure}[h!]
   \centering    
\begin{tikzpicture}[line cap=round,line join=round]
   % dimensions
   \def\a{1.5}
   \def\b{1}
   \def\h{4}
   \def\z{4}
   \def\w{2.2}
   % cylinders 
   \foreach\i in {0} % we draw the same figure at heights 0 and \h
   {%

%%%%Small cylinder
\filldraw[dashed,color=red!100, fill=red!20, very thick] (\z,\i+\b) 
   -- (0,\i+\b) arc (90:270:0.5*\b cm and \b cm) -- (\z,\i-\b) ;
\filldraw[dashed,color=red!100, fill=red!20, very thick] (0,\i-\b) 
   arc (-90:90:0.5*\b cm and \b cm);
\filldraw[color=red!100, fill=red!20, very thick]  (\z,\i) ellipse 
   (0.5*\b cm and \b cm);

%%%% Big cylinder
\draw[black] (-0*\a,\i) -- (\z,\i); % axis
\draw (\z,\i+\a) -- (0,\i+\a) arc (90:270:0.5*\a cm and \a cm) -- 
   (\z,\i-\a) ;
\draw[dashed] (0,\i-\a) arc (-90:90:0.5*\a cm and \a cm);
\draw  (\z,\i) ellipse (0.5*\a cm and \a cm);
\draw[black,-stealth] (\z,\i) -- (1.5*\z,\i) ; % axis
   \node[label={$\xi_1$}] at (1.5*4,0) {};
   \fill[black] (0,\i)  circle (1pt);
   \fill[black] (\z,\i) circle (1pt);
\draw[black,-stealth] (0,0) -- (0,3) ;
   \node[label={$\xi_2$}] at (0.4,2.4) {};
\draw[black,-stealth] (0,0) -- (-2,-2) ;
   \node[label={$\xi_3$}] at (-2,-2) {};
\draw[black](4,-0.5)--(5, -1.2);
%\draw[black] (-0.5*\z,\i) -- (-0.5*\a,\i); 
\draw (4,0.3) node { $L$};
\draw (-0.2,0.2) node { $0$};
\draw (2,0.4) node {\Large $\Omega_c$};
\draw (2.6,2) node {\Large $\Omega$};  
\draw (5.3,-1.4) node {\Large $\Gamma_0$}; }
% auxiliary lines and labels
\end{tikzpicture}  

\caption{The domain $\Om\in\rline^3$ is a cylinder and the liquid 
metal is in the smaller cylinder $\Om_c$. The two disks at the ends 
that together are the region $\Gamma_0$ of the boundary, connect to 
chambers filled with the same liquid metal, and these are connected 
to the electric terminals of this device.} \label{fig:domain}
\end{figure}
%%%%%%%%%%**********%%%%%%%%%%**********%%%%%%%%%%**********%%%%%%%%%%

It can be easily checked that the boundary $\Gamma$ of $\Om$ is
Lipschitz. We decompose the boundary as $\Gamma=\overline{\Gamma}_0
\cup\overline{\Gamma}_1$, such that $\Gamma_0\cap\Gamma_1=\emptyset$.
The boundary $\Gamma_0$ denotes the open discs of radius $R_c$ at
$\xi_1=0$ and at $\xi_1=L$ which constitute the planar part of the
boundary of $\Om_c$. $\Gamma_1$ is also open. In the region where
$\xi_1<0$ there is a chamber containing the same liquid metal, and
this chamber is in contact (over a large surface) with one terminal of
this device. Similarly, the region with $\xi_1>L$ contains another
such chamber that connects to the other terminal. We assume that the
liquid metal is a good electric conductor, so that the tangential
component of the electric field on $\Gamma_0$ is practically zero. We
assume that the input enters the system through $\Gamma_1$ and the
output is observed also on $\Gamma_1$.
 
We assume there are no external sources of electric or magnetic field,
so that the electric field intensity $\EEE$ and the magnetic field
intensity $\HHH $ evolve according to Maxwell's equations:
\begin{eqnarray} \label{Maxwelleqs}
   \left\{\begin{array}{ll}\frac{\partial \BBB}{\partial t} \m=\m 
   -\rot\EEE, \qquad \frac{\partial\DDD}{\partial t} \m=\m -\JJJ + 
   \rot\HHH, \\ \div\BBB \m=\m 0, \hspace{16mm} \div\DDD \m=\m \rho 
   \m, \end{array}\right.
\end{eqnarray}
where $\rho$ is the charge density and $\JJJ$ is the current density.
The electric flux density $\DDD$ depends linearly on $\EEE$ and the 
magnetic flux density $\BBB$ depends linearly on $\HHH$. The current
density $\JJJ$ depend linearly on $\EEE$. The relationships are as
follows\m: \vspace{-2mm}
\begin{equation} \label{Merkava}
   \DDD \m=\m \e\EEE\m, \qquad \BBB \m=\m \mu\HHH \m, \qquad 
   \JJJ \m=\m g\EEE\m,
\end{equation}
where the constant $\e>0$ is the electric permittivity and the 
constant $\mu>0$ is the magnetic permeability.
The function $g\geq 0$ is the conductivity of the material, which 
depends on time and location: Under normal operating conditions, 
we have $g(x)=g_0$ for $x\in\Om_c$, and $g(x)=0$ for 
$x\in\Om\setminus\Om_c$. Here $g_0>0$ is the conductivity of the 
liquid metal (which is 16 times less than that of copper according to
\cite{Yang}).

We denote by $\nu\in L^{\infty}(\Gamma,\rline^3)$ the outward unit
normal vector field on $\Gamma$ and $\gamma_0\in\Lscr(\Hscr^1(\Om;
\rline^3),\Hscr^{\frac{1}{2}}(\Gamma;\mathbb{R}^3))$ stands for the
Dirichlet trace operator, see for instance \cite{Lions} and 
\cite{WeSt13}. Thus, for sufficiently smooth fields $\EEE$, on the 
boundary $\Gamma$ we have that $\gamma_0\EEE=\EEE$. 

If $\EEE\in\Hscr^1(\Om;\rline^3)$, then we denote the {\it tangential
component of the Dirichlet trace of} $\EEE$ by 
\begin{equation*}
   \pi_{\tau} \EEE =( \nu \times \gamma_0 \EEE) \times \nu = \gamma_0 
   \EEE -(\gamma_0 \EEE \cdot \nu) \nu,
\end{equation*}
so that 
$$ \pi_{\tau}\in\Lscr(\mathcal{H}^1(\Om;\rline^3),L^2(\Gamma;
   \rline^3)) \mbox{  and  } \pi_{\tau} \EEE \cdot \nu =0.
  $$
Let $r$ be a positive scalar function of $\Gamma_1$ such that $r,\
r^{-1}\in L^\infty(\Gamma_1)$. Now we define the boundary control 
with input $u$ of this system as follows\m: For $t\geq 0$
\begin{equation} \label{input_BC}
   \begin{cases} \frac{1}{\sqrt{2}} \left(r \nu \times \gamma_0 
   \HHH+\pi_{\tau} \EEE \right)=u \quad {\rm on}\ \ \Gamma_1\m,\\
   \m\quad \pi_\tau\EEE=0  \quad \hspace{2.42cm} {\rm on}\ \ 
   \Gamma_0\m. \end{cases}
\end{equation}
The state space $X$, the input space $U$ and the output space $Y$ are
as follows:
\begin{eqnarray*}\label{statespace}
   \begin{array}{ll} X \m=\m\Escr\times \Escr, \quad\quad \Escr=L^2
   (\Om;\rline^3),\\ Y \m=\m U \m=\m \left\{w\in L^2(\Gamma_1;
   \rline^3)\ |\ w \cdot \nu=0\right\}, \end{array} 
\end{eqnarray*}
and the state at time $t\geq 0$ is $x(t)=[\BBB(t) \ \DDD(t)]^{\top}$.
We define the inner product on $X$ by
\begin{equation} \label{ip_X}
   \left\langle\bbm{\BBB\\ \DDD},\bbm{\tilde{\BBB}\\ \tilde{\DDD}}
   \right\rangle_X \m=\m \frac{1}{\mu}\langle \BBB, \tilde{\BBB}
   \rangle_{L^2} + \frac{1}{\e}\langle \DDD, \tilde{\DDD}
   \rangle_{L^2}\m. 
\end{equation}
On $U$, we use the inner product\vspace{-1mm}
\begin{equation} \label{ip_U}
   \langle u,\tilde{u}\rangle_U \m=\m \int_{\Gamma_1} \frac{1}{r(\xi)}
   u(\xi)\tilde{u}(\xi) \dd\xi \m.\vspace{-1mm}
\end{equation}
We denote by $\Escr_0$ the following subset of $\Escr$\m:
\vspace{-1mm}
$$ \Escr_0 \m=\m \left\{ \EEE\in L^2(\Om;\rline^3)\ |\ \rot\EEE
   \in L^2(\Omega;\rline^3),\m \pi_{\tau}\EEE \in L^2(\Gamma;
   \rline^3), \pi_{\tau}\EEE|_{\Gamma_0}=0 \right\},\vspace{-1mm}$$
endowed with the norm
$$ \|\EEE\|^2_{\Escr_0}\m=\m \|\EEE\|^2_{\Escr}+\|\rot\EEE\|^2
   _{\Escr}+\|\pi_\tau\EEE\|^2_U.$$
It follows from the above definition that $\Escr_0$ is dense in
$\Escr$ with continuous embedding, for details we refer to
\cite[Sect. 4]{WeSt13}. The operator $\pi_{\tau}$ can be extended such
that $\pi_{\tau}\in\Lscr(\Escr_0,U)$ and we denote by $\left(
\pi_{\tau} \right)^* \in \mathcal{L}(U,\Escr_0')$, where $\Escr_0'$ is
the dual of $\Escr_0$ with respect to the pivot space $\Escr$.

Maxwell's equation \rfb{Maxwelleqs} can be rewritten as follows\m: 
For $t\geq 0$,\vspace{-1mm}
\begin{equation} \label{MaxwellStateEqn_old}
   \bbm{\dot{\BBB}(t)\\ \dot{\DDD}(t)} \m=\m \bbm{0 & -L\\L^* & 
   Q-gI} P\bbm{\BBB(t)\\ \DDD(t)} \m, \vspace{-1mm}
\end{equation}
where $L=\rot$, with $\Dscr(L)=\Escr_0$. Here 
$$P \m=\m \sbm{\mu^{-1} & 0\\ 0 & \e^{-1}},\qquad Q \m=\m
    -(\pi_{\tau})^* r^{-1} \pi_{\tau}\m.$$
Thanks to the recent article \cite{Nathanael2}, we know that the
domain $\Dscr([\AB])$ is \vspace{-1mm}
\begin{equation} \label{domainA}
   \Dscr([\AB]) \m=\m \left\{ \bbm{\mu\HHH\\ \e\EEE\\ w}\n\in
   \hspace{-2mm} \left.\begin{array}{c}\Escr\vspace{-1mm} \\ \times 
   \\ \Escr\vspace{-1mm}\\\times\\ U\end{array}\right| 
   \begin{array}{c}{\rm rot}\HHH\in \Escr,\ \ \EEE\in \Escr_0,
   \vspace{1mm}\\  \frac{\sqrt{2}}{r}w = \left.\left(\nu \times 
   \gamma_0 \HHH\right)\right|_{\Gamma_1}\\ \vspace{2mm}
   \hspace{10mm}+\frac{1}{r}\left.\left(\pi_{\tau} \EEE\right)
   \right|_{\Gamma_1} \end{array}\n\right\}, \vspace{-1mm}
\end{equation} 
where we have used the relation \rfb{Merkava}. (In \cite{WeSt13} it
was only proved that the left side is a subset of the right side. We
refer to Remark 5.7 in \cite{WeSt13} and the comments around it for
more details on \rfb{domainA}). For any $[\mu\HHH \ \e\EEE \ w]^\top
\in\Dscr([\AB])$, 
\begin{equation} \label{ProperDef}
   [\AB]\bbm{\mu\HHH\\ \e\EEE\\w} =\m \bbm{-{\rm rot}\m\EEE\\
   {\rm rot}\m\HHH -g\EEE} \m.
\end{equation}

The output is described by the output equation in \rfb{Ax+Bu}, where
$[\CD]$ (a boundary trace operator) is defined on $\Dscr([\AB])$ as
follows\m: \vspace{-1mm}
\begin{equation} \label{eq:max_output}
   [\CD]\nm\bbm{\mu\HHH\\ \e\EEE\\ w} \m=\m \frac{1}{\sqrt{2}} 
   \left(r\left.\left(\nu \times \gamma_0 \HHH\right)
   \right|_{\Gamma_1}-\left.\left(\pi_{\tau} \EEE\right)
   \right|_{\Gamma_1} \right). \vspace{-1mm}
\end{equation}
Thus, the linear system \rfb{MaxwellStateEqn_old} with the boundary
condition \rfb{input_BC} is equivalent to the system described by
$\dot x=[\AB]\sbm{x\\ u}$, with $[\AB]$ as above. 

It follows from \cite[Theorems 5.1, 5.4 \& 5.6]{WeSt13} that the
Maxwell's equations system in \rfb{input_BC}, 
\rfb{MaxwellStateEqn_old} and $y=[\CD]\sbm{x\\ u}$, with $[\CD]$ as in
\rfb{eq:max_output}, is a linear scattering passive system (if we use
the inner products from \rfb{ip_X} and \rfb{ip_U}). 

The average current through this device is obtained by computing 
first the total current through a disk of radius $a$ in the cross 
section of the device at a given position $\xi_1\in(0,L)$: From the 
equation $\frac{\partial\DDD}{\partial t}=-\JJJ+\rot\HHH$ (which is
a part of \rfb{Maxwelleqs}), using Stokes' theorem, this total current
(conduction current plus displacement current) is\vspace{-2mm}
$$ I(a,\xi_1) \m=\m \int_{\Gamma_{a,\xi_1}} \HHH \cdot \dd x \m,$$
where the curve $\Gamma_{a,\xi_1}$ is a circle of radius $a$, the
boundary of the disk mentioned a little earlier. If $a\in[R_c,R]$,
then $I(a,\xi_1)$ is practically equal to the current flowing through
the device at the cross-section $\xi_1$, because all the conduction
current flows through $\Om_c$ (which has radius $R_c$), and a
negligible displacement current flows outside of $\Om_c$, so that 
$I(a,\xi_1)$ is almost a constant as long as $a\in[R_c,R]$. Similarly,
since there is no electric charge inside $\Om_c$, from div $J=\frac
{\dd\rho}{\dd t}$ (where $\rho={\rm div}\m D$ is the charge density),
the conduction current $J$ is practically independent of $\xi_1$, so 
that $I(a,\xi_1)$ is almost constant as long as $\xi_1\in[0,L]$. The
above functional $I(a,\xi_1)$ is, unfortunately, not a bounded linear
functional on $L^2(\Om;\rline^3)$. To eliminate this deficiency, we
average it over all $a\in[R_c,R]$ and over all $\xi_1\in[0,L]$ (we 
refrain from writing the integral explicitly), using that in this 
region it is anyway almost a constant. We call the resulting average 
$i(\HHH)$ the average current. The current $i(\HHH)$ is a bounded 
linear functional of the field $\HHH\in L^2(\Om;\rline^3)$.

We remark that an alternative approach to define the average current
through this device is obtained by first computing
the current through a cross-section for one specific $\xi_1\in(0,L)$, 
by integrating $\JJJ_1=g\EEE_1$, the current density in the $\xi_1$ 
direction, over that cross-section. Then we could average this current
over all possible values of $\xi_1\in[0,L]$. The resulting expression 
$i(\EEE)$ is a bounded linear functional of the field $\EEE\in L^2
(\Om;\rline^3)$, but it is also proportional to $g$ which is a big 
drawback, because this device changes its $g$ when it reaches the 
current limit, which would complicate the analysis.

If we restrict the average current $i(\HHH)$ defined ealier such that 
$|i(\HHH)|\leq i_{\rm max}$, where $i_{\rm max}>0$, then the 
corresponding set $K=K_0\times\Escr$, where \vspace{-1mm}
\begin{equation} \label{ConvexSet}
   K_0 \m= \left\{\HHH \in \Escr \m\m\big|\ |i(\HHH)| \m\leq\m 
   i_{\rm max} \right\}\m,\vspace{-1mm}
\end{equation}
is closed and convex in $X$. We are in the situation described in 
Subsection 3.3 (one dimensional constraint), now $i(\HHH)$ plays the 
role of $z_1$ from \rfb{Pibis}, and the generator of the Lax-Phillips 
semigroup of our system plays the role of $\GothA$ in \rfb{Pibis}. On 
$K$ the projected dynamical system $\Sigma^K$ associated to 
\rfb{domainA}, \rfb{ProperDef} and \rfb{eq:max_output} is given by the 
differential inclusion and the output equation in \rfb{Ax+Bu_K}. It
follows from Theorem \ref{pro_Lax} that this is an incrementally 
scattering passive system.

%%%%%%%%%%**********%%%%%%%%%%**********%%%%%%%%%%**********%%%%%%%%%%
\begin{figure}[h!] % Figure 3
   \centering 
   \includegraphics[height=45mm]{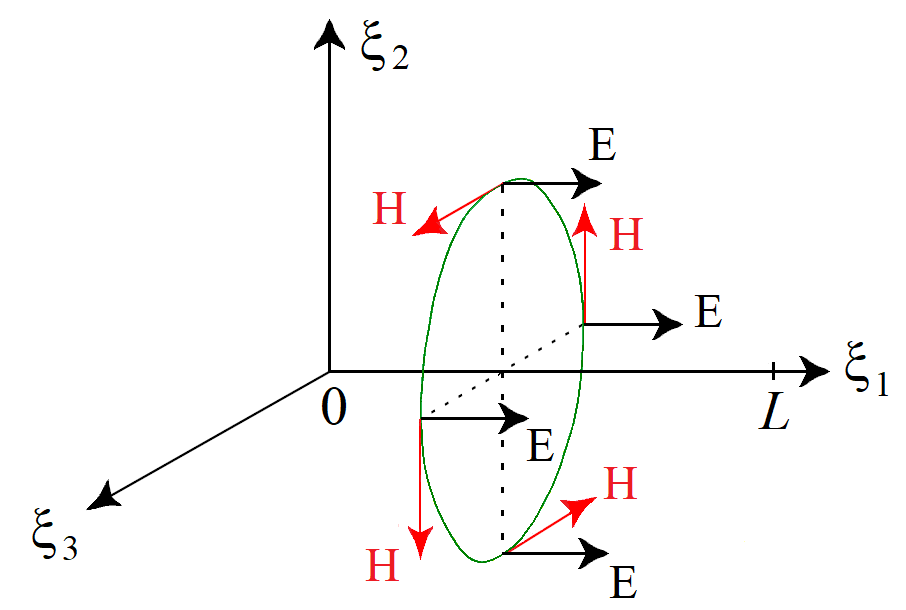}
   \caption{The fields $\EEE$ and $\HHH$ shown in four points on a 
   circle (indicated in green) that in the part $\Gamma_1$ of the 
   boundary, when the current $i$ is flowing to the right along the 
   $\xi_1$ axis.} \label{EM_fields}
\end{figure}

%%%%%%%%%%**********%%%%%%%%%%**********%%%%%%%%%%**********%%%%%%%%%%

We give some engineering interpretation of the input $u$ and the 
output $y$ of \m $\Sigma$. We have seen that $u$ satisfies
\rfb{input_BC}. According to \rfb{eq:max_output}, $y$ is given by
\vspace{-1mm}
$$y \m=\m \frac{1}{\sqrt{2}}\left(r\nu\times \gamma_0\HHH -
    \pi_\tau\EEE\right) \qquad {\rm on} \ \ \Gamma_1 \m.\vspace{-1mm}$$
If there are no significant electromagnetic fields originating from
outside this system, then the electric field along $\Gamma_1$ is
created by the voltage $V$ between the terminals of our device, which
practically equals the voltage between the two ends of the channel
$\Om_c$. Thus, $\EEE$ is uniform and pointing in the $\xi_1$
direction. Similarly, according to the second equation in
\rfb{Maxwelleqs}, and neglecting the term $\frac{\partial\DDD}
{\partial t}$, the magnetic field around the conducting cylinder
$\Om_c$ will be caused by the current $i=i(\HHH)$ and it will by 
tangent to a circle around the $\xi_1$ axis. The line integral of 
$\HHH$ along this circle must be $i$, as already discussed earlier.
Due to circular symmetry around the $\xi_1$ axis, it will be enough
to express the approximate values of $\EEE$ and $\HHH$ on the lines
in $\Gamma_1$ where $\xi_2=0$: \vspace{-2mm}
$$ \EEE \m\approx\m \bbm{V/L\\ 0\\ 0} \nm\qquad
   \HHH \m\approx\m \bbm{0\\ \frac{-i}{2\pi \xi_3}\\ 0} 
   \m,\vspace{-2mm}$$
see Fig.~\ref{EM_fields}. It is interesting to note that, in our 
above approximation of the fields, $\nu\times\HHH$ is always pointing
in the $\xi_1$ direction, just like $\EEE$. Thus, both $u$ and $y$ 
are (at every point in $\Gamma_1$) vectors pointing in the $\xi_1$ 
direction, so that their second and third components are zero. The 
approximate values of the first components on the round part of 
$\Gamma_1$ are\m:\vspace{-1mm}
$$ u_1 \m=\m \frac{1}{\sqrt{2}}\left( \frac{ri}{2\pi R} +
   \frac{V}{L} \right) \m,\qquad y_1 \m=\m \frac{1}{\sqrt{2}}\left( 
   \frac{ri}{2\pi R} - \frac{V}{L} \right) \m.$$
These variables are linear combinations of $V$ and $i$ that only 
differ in the sign of the $V$ term. Such variables are called 
scattering variables in electric circuit theory.

What happens at the physical level when $|i(\HHH)|$ reaches the limit 
$i_{\rm max}$? A cavity (without liquid metal) is formed in the 
channel $\Om_c$ and the electric current passes this cavity by forming
an arc. This arc has high resistance and high inductance. Thus, the 
voltage on the FCL increases, and the device works approximately as a 
current source. After the current returns below the saturation limit, 
the device returns to normal operation (instead of burning out like a 
fuse). 

%%%%%%%%%%**********%%%%%%%%%%**********%%%%%%%%%%**********%%%%%%%%%%
%\bibliographystyle{plain}        % Include this if you use bibtex 
%\bibliography{bibfile}  

\end{document}